\documentclass[12pt]{amsart}
\usepackage{etex}

\usepackage{amssymb,amsthm,amsmath,amsfonts,stmaryrd,soul}

\usepackage{bbm}
\usepackage[utf8x]{inputenc}

\usepackage{chngcntr}
\counterwithin{figure}{section}

\makeatletter{}\usepackage[OT2,OT1]{fontenc}
\newcommand\cyr
{
\renewcommand\rmdefault{wncyr}
\renewcommand\sfdefault{wncyss}
\renewcommand\encodingdefault{OT2}
\normalfont
\selectfont
}
\DeclareTextFontCommand{\textcyr}{\cyr}

\newcommand{\stirling}[2]{\genfrac{\{}{\}}{0pt}{}{#1}{#2}}

\usepackage{mathabx}
\usepackage{todonotes}
\newlength{\tmpl}

\usepackage[colorlinks,anchorcolor=black,citecolor=black,linkcolor=black,filecolor=black,urlcolor=black]{hyperref}

\usepackage{color}

\usepackage{pict2e}

\usepackage[margin=2cm,a4paper]{geometry}
\usepackage[all]{xypic}

\usepackage[extra]{tipa}
\usepackage{graphicx}

\usepackage{pst-node}
\usepackage{enumerate}
\usetikzlibrary{arrows}
\numberwithin{equation}{section}

\newcommand{\Dyck}{\mathcal{D}}
\newcommand{\dyckfrac}[2]{\genfrac{}{}{0pt}{0}{#1}{#2}}

\DeclareMathOperator{\Tr}{\mathrm{Tr}}

\DeclareMathOperator{\cotsum}{S}

\DeclareMathOperator{\Paths}{\mathrm{Paths}}

\let\Re\undefined
\let\Im\undefined
\DeclareMathOperator{\Re}{\mathrm{Re}}
\DeclareMathOperator{\Im}{\mathrm{Im}}
\DeclareMathOperator{\arccot}{\mathrm{arccot}}
\DeclareMathOperator{\atanh}{\mathrm{atanh}}

\newcommand{\harpleftsign}{\scriptstyle\leftharpoonup}
\newcommand{\harpleft}[2]{  \ifx\displaystyle#1\doalign{$\harpleftsign$}{#1#2}\fi
  \ifx\textstyle#1\doalign{$\harpleftsign$}{#1#2}\fi
  \ifx\scriptstyle#1\doalign{\scalebox{.6}[.9]{$\harpleftsign$}}{#1#2}\fi
  \ifx\scriptscriptstyle#1\doalign{\scalebox{.5}[.8]{$\harpleftsign$}}{#1#2}\fi
}

 \newcommand{\harprightsign}{\scriptstyle\rightharpoonup}
\newcommand{\harpright}[2]{  \ifx\displaystyle#1\doalign{$\harprightsign$}{#1#2}\fi
  \ifx\textstyle#1\doalign{$\harprightsign$}{#1#2}\fi
  \ifx\scriptstyle#1\doalign{\scalebox{.6}[.9]{$\harprightsign$}}{#1#2}\fi
  \ifx\scriptscriptstyle#1\doalign{\scalebox{.5}[.8]{$\harprightsign$}}{#1#2}\fi
}

\newcommand{\doalign}[2]{ {\vbox{\offinterlineskip\ialign{\hfil##\hfil\cr#1\cr$#2$\cr}}}}

\def\M{\tilde{M}}
\def\C{{\mathbb C}}

\def\IZ{{\mathbb Z}}
\def\IN{{\mathbb N}}
\def\N{{\mathbb N}}

\def\<{\langle}
\def\>{\rangle}

\def\c{\underline c}

\def\J{J_n}

\newcommand{\abs}[1]{
  \left\lvert
    #1
  \right\rvert}

   \DeclareMathOperator{\SP}{\mathcal{P}} \DeclareMathOperator{\SPodd}{\mathcal{P^{\textrm{odd}}}}

\newtheorem{th-def}{Theorem-Definition}[section]
\newtheorem{theo}{Theorem}[section]
\newtheorem{lemm}[theo]{Lemma}

\newtheorem{prop}[theo]{Proposition}
\newtheorem{cor}[theo]{Corollary}
\theoremstyle{definition}

\newtheorem{defi}[theo]{Definition}

\newtheorem{exam}[theo]{Example}
\newtheorem{Rem}[theo]{Remark}

\def\cvput#1[#2]{\pnode(#1,1){#1} \pscircle*(#1,1){.1} \rput(#1,.5){$#2$}}

\tikzset{
  treenode/.style = {align=center, inner sep=0pt, text centered,
    font=\sffamily},
  arn_n/.style = {treenode, circle, white, font=\normalfont, draw=black,
    fill=black, text width=.25em},      arn_m/.style = {treenode, circle, white, font=\normalfont, draw=black,
    fill=black, text width=.65em},  arn_r/.style = {treenode, circle, red, draw=red, 
    text width=.65em, very thick},     arn_b/.style = {treenode, circle, blue, draw=blue, 
    text width=.65em, very thick},  arn_x/.style =  {treenode,dashed, circle, black, font=\normalfont, draw=black,
    fill=white, text width=3em}}

\usetikzlibrary{calc,shapes.geometric} 
\def\drawconnectvertices[num vertex=#1, circle radius=#2] at (#3);{\pgfmathtruncatemacro\vertices{#1}
\pgfmathsetmacro\circleradius{#2}
\pgfmathsetmacro\halfcircleradius{\circleradius/2}
\draw[blue] (#3) circle (\halfcircleradius cm) node[regular polygon, regular polygon sides=\vertices, minimum size=\circleradius cm, draw=none, name={vertex set}] {};
\foreach \x in {1,...,\vertices}{
    \node[draw,circle, inner sep=1pt,blue, fill=blue] at (vertex set.corner \x) {};
}

   }

\title{The Trace Method for Cotangent Sums}
\author[Wiktor Ejsmont]{Wiktor Ejsmont}

\address[Wiktor Ejsmont]{ 
Instytut Matematyczny, Uniwersytet Wroc\l awski,\\ pl. Grunwaldzki 2/4, 50-384 Wroc\l aw, Poland}
\email{wiktor.ejsmont@gmial.com}
\author[Franz Lehner]{Franz Lehner}
\address[Franz Lehner]{Institut f\"ur Diskrete Mathematik,
TU Graz
Steyrergasse 30, 8010 Graz, Austria }
\email{lehner@math.tugraz.at}

\thanks{Supported by the Austrian Federal Ministry of Education, Science and
  Research and the Polish Ministry of Science and Higher Education, grant
  N$^{\textrm{o}}$  PL 06/2018  and Wiktor
Ejsmont was supported by the Narodowe Centrum Nauki grant no. 2018/29/B/HS4/01420}

\begin{document}

\begin{abstract} 
This paper presents a combinatorial study of sums of integer powers of
the cotangent which is a popular theme in classical calculus.
Our main tool the realization of cotangent values as eigenvalues of a simple self-adjoint matrix with integer matrix.
We use the trace method to draw conclusions about integer values of the sums
and and provide explicit evaluations;
it is remarkable that throughout the calculations the combinatorics are governed by the higher tangent and arctangent numbers exclusively. 
Finally we indicate a new approximation of the values of the Riemann zeta function at  even integer arguments.  
\end{abstract}

\date{
\today}

\setlength{\parindent}{0pt}

\maketitle

\tableofcontents{}

\section{Introduction}

It is a well-known fact that the trace of a matrix equals the sum of its eigenvalues
$$
\Tr A = \sum \lambda_i
,
$$
counting algebraic multiplicities.
This relation is respected by functional calculus and the identity 
$$
\Tr f(A) = \sum f(\lambda_i)
$$
holds for arbitrary holomorphic functions, 
in particular, powers and polynomials. 
The \emph{trace method} consists in the evaluation
of this identity for particular matrices in order to obtain nontrivial combinatorial relations.

In the present paper we apply this method to
cotangent sums of the form
\begin{equation}
\label{eq:cotsum}
\cotsum(m,n,\alpha) = 
\sum_{k=0}^{n-1} \cot^m\frac{\alpha+k\pi}{n}
\end{equation}
for $\alpha\neq k\pi$, $n,m\in \N$, $n\geq 2$,
and the limit case
\begin{equation}
\label{eq:cotsum0}
S_0(m,n) = 
\sum_{k=1}^{n-1} \cot^m\frac{k\pi}{n}
.
\end{equation}
Finite cotangent sums
are a recurrent theme in the mathematical literature.
They arise in number theory in connection with Dedekind sums
and topology
\cite{Zagier:1973,HirzebruchZagier:1974,Sofo:2018:euler}, 
and more recently were used to evaluate the Riemann zeta function,
see \cite[Problem 141ff]{Yaglom:1967:challenging2} for the apparently first
occurrence of this connection, later rediscovered in
\cite{Holme:1970,Williams1971,Papadimitriou:1973,Bencze:1980:aufgabe828,Apostol1973,Cvijovic2003}; 
Berndt and Yeap 
\cite{Berndt2002} attribute the first occurrence of cotangent sums
to \cite[p.~155]{Stern:1861}. 
The recent literature on this topic is abundant, 
in particular concerning reciprocal relations
\cite{Berndt2002,Chu:2018:reciprocal}
and
the question for which values of the parameters
the sums \eqref{eq:cotsum} yield integer values is intriguing.
For example, Byrne and Smith \cite{ByrneSmith:1997:integer} 
proved that the sums are integer valued polynomials in $n$
at the offset $\alpha=\pi/4$, found the leading terms 
and established recurrence relations.
We were led to study such sums in connection with certain
limit theorems arising in free probability, see \cite{EjsmontLehner:2019:commutators,EjsmontLehner:2020:tangent}, 
where the matrices considered below arise in a natural way.

The most general expression for the sum \eqref{eq:cotsum}  so far was obtained
by Cvijovi\'c and Klinowski  \cite{CvijovicKlinowski:2000},
who realized  the cotangent values $\cot\frac{\alpha+k\pi}{n}$ 
as roots of a polynomial
and expressed the sums via Cramer's rule applied to the Newton relations
between elementary and power sum symmetric functions. 
In the present paper we go one step further and show that the polynomial found
in \cite{CvijovicKlinowski:2000} is in fact the characteristic polynomial 
of a simple matrix. 
Thus  the trace method is applicable and we can draw certain conclusions about the sum \eqref{eq:cotsum}. 
For example, if $\cot\alpha$ is an integer,
say, $\alpha=\pi/4$, it follows trivially that
\eqref{eq:cotsum} evaluates to an integer,
as was observed by different means in \cite{ByrneSmith:1997:integer}. 
Moreover we obtain explicit expressions in terms of tangent numbers and
derivative polynomials by extrating Taylor coefficients from
suitable generating functions.

The main results can be summarized as follows.
First,  the cotangent sum
\eqref{eq:cotsum} is the trace of the $m$-th power of the matrix \eqref{eq:Cn}.
Consequently it is a polynomial in $\cot\alpha$ with integer coefficients 
and it follows immediately that the sum is integer valued whenever
$\cot\alpha$ is an integer.
Moreover, the sums can be expressed in terms of  tangent and arctangent numbers.
In the simplest case ($\cot\alpha=0$) the odd power sums vanish
and the even ones evaluate to
\begin{equation*}
    \cotsum(2m,n,\pi/2) 
     = (-1)^{m}n
    + \frac{1}{(2m-1)!}   \sum_{k=1}^m n^{2k} A_{2m}^{(2k)}\,T_{2k-1}
    .
\end{equation*}
For general $\alpha$,
the coefficients of the polynomial
\begin{equation*} 
  \cotsum(m,n,\alpha) =\sum_{0 \leq k\leq \lfloor m/2\rfloor }p_{m,m-2k}(n)\cot^{m-2k}\alpha.
\end{equation*}
can be expressed in terms of tangent and arctangent numbers as well
\begin{equation*}
  p_{m,r}(n) =
  \begin{cases}
     \cotsum(m,n,\pi/2) & r=0\\
     \displaystyle
     \frac{1}{r(m-1)!}\sum_{k=r}^m n^kA_m^{(k)}T_k^{(r)} & 1\leq r\leq m
  \end{cases}
\end{equation*}
and as special cases we recover the sums
\begin{align*}
\cotsum(2m+1,n,\pi/4)&=\sum_{k=1}^n(-1)^{k}\cot^{2m+1}\frac{(2k-1)\pi}{4n} 
= \frac{1}{2(2m)!}   \sum_{k=0}^m (2n)^{2k+1} A_{2m+1}^{(2k+1)} \, S_{2k},
 \\
\cotsum(2m,n,\pi/4) &=\sum_{k=1}^n\cot^{2m}\frac{(2k-1)\pi}{4n}
= (-1)^mn
  + \frac{1}{2(2m-1)!}   \sum_{k=1}^{m} (2n)^{2k} A_{2m}^{(2k)} \, T_{2k-1},
\end{align*}
of Byrne and Smith \cite{ByrneSmith:1997:integer}
in terms of secant, tangent, and arctangent numbers,
see Corollary~\ref{cor:cotsumpi/24}.
Finally we obtain an explicit formula for the sum  \eqref{eq:cotsum0}
$$
      \sum_{k=1}^{n-1} \cot^{2m}\frac{k\pi}{n}
      = (-1)^m (n-1) - \frac{1}{(2m-1)!}\sum_{k=1}^{m}  (-1)^kA_{2m}^{(2k)} \frac{4^kB_{2k}}{2k}(n^{2k}-1).
$$
which was previously evaluated by Berndt and Yeap
in terms of Bernoulli numbers \cite{Berndt2002}
(cf.~also \cite{Williams1994,Gessel:1997,Cvijovic2003,Annaby2011,Fonseca2017,He:2020}), see Corollary~\ref{cor:BerndtYeap}. 
Chu and Marini \cite{ChuMarini:1999:partial} wrote a systematic
study of generating functions and we complement this
in Section~\ref{sec:genfunc} by providing a generating function
for arbitrary  $\alpha$.

Finite sums of trigonometric functions are a popular subject
in terms of generating functions
 \cite{ChuMarini:1999:partial} and 
reciprocal relations
 \cite{Berndt2002,Chu:2018:reciprocal}.
For explict evaluations of 
sums of powers of cosines see
\cite{DaFonsecaKowalenko:2013,LvShen:2017:chebyshev},
for sines see
\cite{Holcombe:2018}
and for secants and cosecants see
\cite {GrabnerProdinger:2007,FonsecaGlasserKowalenko:2018:cosecant}.
For an evaluation of cosecant sums via the trace method see
\cite {StembridgeTodd:1981};
for the exponent $m=2$ finite Fourier analysis is applicable
\cite{BeckHalloran:2010}.

It is perhaps interesting to note that the papers 
\cite{CalogeroPerelomov:1978:properties,CalogeroPerelomov:1979:diophantine}
evaluate certain trigonometric sums using matrices with trigonometric entries
and integer eigenvalues, while in the present paper we exploit integer matrices
with trigonometric eigenvalues.

\emph{Acknowledgements}.
We are grateful to an anonymous referee who brought references
\cite{AndelicDuFonseca:2020,FonsecaGlasserKowalenko:2018:cosecant,LvShen:2017:chebyshev,Sofo:2018:euler} to our attention.

\section{Preliminaries on Linear Algebra and the Tangent function}

The main role in this paper is played by a certain matrix and its intricate
relations to the tangent and cotangent functions. 
\subsection{A matrix}
For scalars $a,b, c\in\C$  we denote by  $\left[\begin{smallmatrix}
    a &b\\
    c& a
    \end{smallmatrix}\right]_n
\in M_n(\C)
$
the matrix whose diagonal elements are equal to $a$, whose  upper-triangular entries are equal to $b$ and
whose lower-triangular elements are equal to $c$, respectively. 
For simplicity of notation, we use the same letter $\J$ and $B_n$ for the following matrices
$$
\J :=
\begin{bmatrix}
  1&1&1&\dots&1\\
  1&1&1&\dots&1\\
  1&1&1&\ddots&1\\
  \vdots&&\ddots&\ddots\\
  1&1&1&\dots&1
\end{bmatrix}
    \qquad
\text{and}
\qquad
B_n := i
\begin{bmatrix}
  0&1&1&\dots&1\\
  -1&0&1&\dots&1\\
  -1&-1&0&\ddots&1\\
  \vdots&&\ddots&\ddots\\
  -1&-1&-1&\dots&0
\end{bmatrix} .
$$
The first observation reveals that the entries of the sum
\eqref{eq:cotsum} can be realized as eigenvalues 
of the following matrix
and consequently the sum is the
trace of the $m$-th power of this matrix.

\begin{lemm}\label{lemm:spectrum} 
If $a=\cot \alpha $, 
then the characteristic polynomials $\chi_n(\alpha;\lambda)$ of the matrices
\begin{equation}
\label{eq:Cn}
C_n= a \J + B_n=
\begin{bmatrix} 
  a  & a+i & \dots  &  a+i  \\ 
  a-i &  a & \dots & a+i
  \\
  \multicolumn{4}{c}{\dotfill}\\
  a-i &  a-i & \dots  &   a
\end{bmatrix}\in M_n(\C)
\end{equation}
satisfy the recurrence relation
\begin{equation}
  \label{eq:chin-recurrence}
  \chi_n(\alpha;\lambda)
  = (2\lambda-2a+w+\widebar{w})\,\chi_{n-1}(\alpha;\lambda)-(\lambda-a+w)(\lambda-a+\widebar{w})\,\chi_{n-2}(\alpha;\lambda)
\end{equation}
and have the following explicit expression
\begin{equation}
  \label{eq:pnlambda}
\chi_n(\alpha;\lambda) 
= \frac{(\cot\alpha+i)(\lambda-i)^n-(\cot\alpha-i)(\lambda+i)^n}{2i}
= \Im(\cot\alpha+i)(\lambda-i)^n
\end{equation}
(assuming $\lambda$ real).
The eigenvalues are given by 
$$
\lambda_k=\cot\frac{\alpha+k\pi}{n}
,\text{ for } 0\leq k\leq n-1. 
$$ 
\end{lemm}

\begin{proof}
The spectrum of the matrix $C_n$ can be computed from its characteristic polynomial
$\chi_n(\alpha;\lambda)=\det (\lambda I-C_n)$ using the following recurrence relation. 
  Let $w=a+i$  , then we have
\begin{align*}
  \chi_n(\alpha;\lambda)
  &=
    \begin{vmatrix}
      \lambda-a &    -w &     -w & -w     &\dots& -w\\
      -\widebar{w}       &\lambda-a &     -w & -w     &\dots&-w\\      
      -\widebar{w}       & -\widebar{w}      &\lambda-a & -w     &\dots&-w\\      
      -\widebar{w}       & -\widebar{w}      & -\widebar{w}      &\lambda-a &\dots&-w\\      
      \multicolumn{6}{c}{\dotfill}\\
      -\widebar{w}       & -\widebar{w}      & -\widebar{w}      &  -\widebar{w}     &\dots&\lambda-a
    \end{vmatrix}
\intertext{we subtract the second row from the first row}
 &=
    \begin{vmatrix}
      \lambda -a+\widebar{w}&  -\lambda  -w+a &     0 & 0     &\dots& 0\\
      -\widebar{w}       &\lambda-a &     -w & -w     &\dots&-w\\      
      -\widebar{w}       & -\widebar{w}      &\lambda -a& -w     &\dots&-w\\      
      -\widebar{w}       & -\widebar{w}      & -\widebar{w}      &\lambda-a &\dots&-w\\      
      \multicolumn{6}{c}{\dotfill}\\
      -\widebar{w}       & -\widebar{w}      & -\widebar{w}      &  -\widebar{w}     &\dots&\lambda-a
    \end{vmatrix}
\intertext{and the   second column from the first column}
 &=
    \begin{vmatrix}
      2\lambda -2a+w+\widebar{w}&  -\lambda  -w+a &     0 & 0     &\dots& 0\\
      -\lambda-\widebar{w}+a       &\lambda-a &     -w & -w     &\dots&-w\\      
      0       & -\widebar{w}      &\lambda-a & -w     &\dots&-w\\      
      0       & -\widebar{w}      & -\widebar{w}      &\lambda-a &\dots&-w\\      
      \multicolumn{6}{c}{\dotfill}\\
      0       & -\widebar{w}      & -\widebar{w}      &  -\widebar{w}     &\dots&\lambda-a
    \end{vmatrix}
  \\
  &= (2\lambda-2a+w+\widebar{w})\chi_{n-1}(\alpha;\lambda)-(\lambda-a+w)(\lambda-a+\widebar{w})\chi_{n-2}(\alpha;\lambda)
\end{align*}
and the solution of this recurrence equation (with initial values
$\chi_0(\alpha;\lambda)=1$
and $\chi_1(\alpha;\lambda)=\lambda$) is
\begin{align*}
  \chi_n(\alpha;\lambda)
  &= \frac{w(\lambda-a+\widebar{w})^n-\widebar{w}(\lambda-a+w)^n}{w-\widebar{w}}\\
&= \frac{(a+i)(\lambda-i)^n-(a-i)(\lambda+i)^n}{2i}=\Im (a+i)(\lambda-i)^n.
\end{align*}
 
  Thus we have to solve the equation 
\begin{equation}
  \label{eq:Imailambda}
\Im (a+i)(\lambda-i)^n=0.
\end{equation}
To compute the zeros,
write $a+i=r_0e^{i\alpha}$, i.e., $a=\cot\alpha$ and assume
$\lambda-i=re^{-i\theta}$. 
Then equation  \eqref{eq:Imailambda} becomes
$\Im r_0e^{i\alpha}r^ne^{-in\theta}=0$ and is equivalent to the equation
$\sin(\alpha-n\theta)=0$, that is, $\alpha-n\theta=-k\pi$
for some $k\in\IZ$. Thus the solutions of   \eqref{eq:Imailambda}
can be written as $\lambda_k=i+r_ke^{-i\theta_k}$ with
$\theta_k=\frac{\alpha+k\pi}{n}$. Now our matrix is self-adjoint,
all roots of the characteristic polynomial \eqref{eq:Imailambda}
are real and hence
$-1=\Im(\lambda_k-i)=-r_k\sin\theta_k$; we conclude
that $r_k=\frac{1}{\sin\theta_k}$
and $\lambda_k = \Re(\lambda_k-i)=r_k\cos\theta_k = \cot\theta_k$.
Consequently
\begin{equation*} 
    \chi_n(\alpha;\lambda) = \prod_{k=0}^{n-1} 
  \left(
    \lambda-\cot\frac{\alpha+k\pi}{n}
  \right).
\end{equation*}
\end{proof}

\makeatletter{}\begin{Rem}
  An alternative formula for this polynomial can be found in
  \cite[Formula (4)]{CvijovicKlinowski:2000}.
Indeed the coefficients of this polynomial are as follows
\begin{align*}
  \chi_n(\alpha;\lambda)
  &= \Im (a+i)(\lambda-i)^n\\
  &= \Im \sum_{k=0}^n\binom{n}{k} (a+i)\lambda^k(-i)^{n-k}\\
  &= \Im \sum_{k=0}^n\binom{n}{k}(a(-i)^{n-k}-(-i)^{n-k+1})\lambda^k\\
  &= \sum_{k=0}^n c_k\lambda^k
\end{align*}
where
$$
c_k =
\begin{cases}
  \displaystyle\binom{n}{k} (-1)^{(n-k)/2} &\text{$n-k$ even}\\[1em]
\displaystyle a\binom{n}{k} (-1)^{(n-k+1)/2} &\text{$n-k$ odd}
\end{cases}
$$
or equivalently,
$$
c_{n-k} =
 \binom{n}{k}
 \left(
   \cos\frac{k\pi}{2} +a \sin\frac{k\pi}{2}
 \right)
 =
\begin{cases}
  \displaystyle \binom{n}{k} (-1)^{k/2} &\text{$k$ even}\\[1em]
  \displaystyle a\binom{n}{k} (-1)^{(k+1)/2} &\text{$k$ odd}
\end{cases}
$$
(cf.~\cite[Formula (4b)]{CvijovicKlinowski:2000}).

In fact the discussion of \cite{CvijovicKlinowski:2000} starts by showing
that the characteristic polynomial $\chi_n(\alpha;x)$ is related to  the expression
$\sin\arccot x$.
Indeed
evaluation of the  polynomial
\eqref{eq:pnlambda} at $\lambda=\cot\theta$
and a few elementary  manipulations 
yield the identity
$$ \chi_n(\alpha;\cot\theta) = \frac{\sin (n\theta-\alpha)}{\sin\alpha \sin^n\theta}
.
$$
\end{Rem}

\begin{Rem}
  The recurrence relation   \eqref{eq:chin-recurrence} 
  falls into the class of recurrence relations with constant coefficients
  which contains the Fibonacci numbers and the Chebyshev polynomials
  \cite {Buschman:1963,Horadam:1965}, see \cite{AndelicDuFonseca:2020}
  for a recent discussion.
\end{Rem}

\makeatletter{}\subsection{Formulas for $\tan(nx)$}
A simple manipulation of the addition formulae for sine and cosine
show that the tangent function obeys the addition rule
\begin{equation}
  \label{eq:tanx+y}
  \tan(x+y) = \frac{\tan x+\tan y}{1-\tan x\tan y}
\end{equation}
This rule is not practical for iteration and the following equivalent
elegant  formula proposed by Szmulowicz \cite{Szmulowicz:2005:analytic}
is a convenient alternative
\begin{align}
  \frac{1+i\tan\sum x_k}{1-i\tan\sum x_k}
  &= \prod\frac{1+i\tan x_k}{1-i\tan x_k}  \nonumber\\
  \intertext{It follows immediately from the identity}
  e^{2ix} &= \frac{1+i\tan x}{1-i\tan x} \nonumber
  \\
  \intertext{and in particular,
    $\tan(n\arctan z)$ is a rational function. Indeed
  }
  \frac{1+i\tan(nx)}{1-i\tan(nx)}
  &=
  \left(
    \frac{1+i\tan x}{1-i\tan x}
  \right)^n
            \label{eq:1+itannx=(1+itanx)n}
\end{align}
and thus
\begin{align*}
  \tan(nx)
  &=i\frac{
    1-
    \left(
      \frac{1+i\tan x}{1-i\tan x}
    \right)^n
  }{    1+
    \left(
      \frac{1+i\tan x}{1-i\tan x}
    \right)^n
  }\\
  &=i\frac{
    (1-i\tan  x)^n - (1+i\tan  x)^n
  }{
    (1-i\tan  x)^n + (1+i\tan  x)^n
  }
\intertext{and}
  \cot(n x) 
  &=i\frac{
    (\cot x + i)^n + (\cot x-i)^n
  }{
    (\cot x + i)^n - (\cot x-i)^n
  }.
\end{align*}
Thus we obtain the well-known formula \cite[item 16]{HAKMEM}
\begin{equation*}
\tan(n\arctan z) = i\frac{(1-iz)^n-(1+iz)^n}{(1-iz)^n+(1+iz)^n}
;
\end{equation*}
comparing with the reciprocal polynomial of
\eqref{eq:pnlambda} at $a=\cot\alpha=0$ which is
$$
\tilde{p}_n(z) = z^n\chi_n(0;1/z) = \frac{(1-iz)^n+(1+iz)^n}{2}
$$

we see that
\begin{equation}
\label{eq:tannatanz=pn}
  \tan(n\arctan z) = -\frac{1}{n+1}\frac{\tilde{p}_{n+1}'(z)}{\tilde{p}_n(z)}
  .
\end{equation}

\subsection{Formulas for $\tan(nx-\alpha)$}
In view of later applications we introduce a nonzero offset into
equation 
\eqref{eq:1+itannx=(1+itanx)n}
and obtain
\begin{equation*}
  \frac{1+i\tan(nx+\alpha)}{1-i\tan(nx+\alpha)}
  =
  \left(
    \frac{1+i\tan x}{1-i\tan x}
  \right)^n
    \frac{1+i\tan \alpha}{1-i\tan \alpha}  
\end{equation*}
which after a few manipulations yields the identity
\begin{align*}
  \tan(nx+\alpha)
  &= i\frac{1-
    \bigl(
    \frac{1+i\tan x}{1-i\tan x}
    \bigr)^n\frac{\cot \alpha+i}{\cot\alpha -i}}{1+
    \bigl(
    \frac{1+i\tan x}{1-i\tan x}
                    \bigr)^n\frac{\cot \alpha+i}{\cot\alpha -i}}
  \\
  &= i\frac{(\cot\alpha-i)(1-i\tan x)^n - (\cot\alpha+i)(1+i\tan x)^n}{(\cot\alpha-i)(1-i\tan x)^n+(\cot\alpha+i)(1+i\tan x)^n}
\end{align*}
The reciprocal provides the following crucial identity for $\cot$
\begin{equation}
  \label{eq:cotnx-alpha}
  \cot(n\arctan z-\alpha) =
  -i\frac{(\cot\alpha+i)(1-iz)^n + (\cot\alpha-i)(1+iz)^n}{(\cot\alpha+i)(1-iz)^n-(\cot\alpha-i)(1+iz)^n}
\end{equation}
which after comparison with the reciprocal polynomial
$$
\tilde{\chi}_n(\alpha;z) = z^n\chi_n(\alpha;1/z) =
\frac{(\cot\alpha+i)(1-iz)^n-(\cot\alpha-i)(1+iz)^n}{2i}
$$
identifies to
\begin{equation*}
  \cot(n\arctan z-\alpha) = \frac{1}{n+1}\frac{\tilde{\chi}_{n+1}'(\alpha;z)}{\tilde{\chi}_n(\alpha;z)}
\end{equation*}

\subsection{Derivatives of $\tan$ and $\cot$}
The higher derivatives of $\tan z$ and $\cot z$ are
closely related, since $\cot z = \tan 
\left(
  \frac{\pi}{2} - z
\right)
$. 
It is easy to see that there exist polynomials $P_n(z)$ such that
$\frac{d^n}{dz^n}\tan z = P_n(\tan z)$; indeed these \emph{derivative polynomials}
satisfy the recursion 
$$
P_{n+1}(x) = (1+x^2)P_n'(x)
$$
and can be used to efficiently compute tangent and Bernoulli numbers
\cite{KnuthBuckholtz:1967}.
Explicitly, these polynomials can be expressed via the \emph{geometric
  polynomials} \cite[(2.1)]{Boyadzhiev:2007}
\begin{equation*}
  \omega_n(x) = \sum_{k=0}^n \stirling{n}{k}k!\,x^k
\end{equation*}
as follows, see \cite[(3.10--11)]{Boyadzhiev:2007}:
\begin{equation}
  \label{eq:tanpoly}
  P_n(z) = (2i)^n(z+i)\,\omega_n
  \left(
    -\frac{iz+1}{2}
  \right)
  =(-2i)^n(z-i)\sum_{k=0}^n \frac{k!}{2^k}\stirling{n}{k} (iz-1)^k
\end{equation}

On the other hand (see  \cite[Lemma~2.1]{Adamchik:2007}  or \cite[(3.15)]{Boyadzhiev:2007})
\begin{equation}
  \label{eq:cotpoly}
  \frac{d^n}{dz^n}\cot z = (-1)^nP_n(\cot z)
  =(2i)^n(\cot z-i)\sum_{k=1}^n\frac{k!}{2^k}
  \stirling{n}{k}
  (i\cot z-1)^k
\end{equation}
and thus $(-1)^nP_n(x)$ serve as derivative polynomials for $\cot$.

Interest in these polynomials goes back at least to Ramanujan
\cite[Chapter~7, entry 11]{Berndt:1985:ramanujan1}
and there is some literature,
see for example
\cite{Schwatt:1962,CarlitzScoville:1972,Williams1994,Hoffman:1995,Hoffman:1999:derivative,Franssens:2007}.

\subsection{Tangent and arctangent numbers}
The \emph{tangent numbers} are the Taylor coefficients of the tangent 
function. They make up the odd part of the sequence of $E_n$ of
 \emph{Euler zigzag numbers}, which are given
by the exponential generating function
\begin{equation}
  \label{eq:def:zigzag}
\tan(z)+\sec(z)=\sum_{n=0}^\infty
\frac{E_{n}}{n!}z^{n}.
\end{equation}
The higher order tangent numbers \cite{CarlitzScoville:1972}
are defined as coefficients of the series
\begin{equation}
  \label{eq:tangentnumbers}
\tan^k z = \sum_{n=k}^\infty \frac{T_n^{(k)}}{n!}\,z^n;
\end{equation}
Their bivariate generating function is
\begin{align*}
  T(x,z)
  &= \sum_{k=1}^\infty x^k \tan^k z \\
  &= \sum_{n=0}^\infty \sum_{k=1}^n\frac{T_n^{(k)}}{n!} x^k z^n \\
  & = \frac{x\tan z}{1-x\tan z}  \\ 
  &= \sum_{n=0}^\infty \frac{T_n(x)}{n!} z^n 
\end{align*}
where $T_n(x)=\sum_{k=1}^{n}T_n^{(k)}x^k$.
On the other hand, from the addition formula   \eqref{eq:tanx+y}
we infer the exponential generating function of the derivative
polynomials to be
\begin{equation}
  \label{eq:derivative:genfunc}
  P(x,z) = \sum_{n=0}^\infty \frac{P_n(x)}{n!}\,z^n = \frac{x+\tan z}{1-x\tan z}
\end{equation}
cf.~\cite[(6.94)]{GrahamKnuthPatashnik:1994} and \cite{Hoffman:1995}.

We can now deduce relations between
these polynomial sequences.
First, by direct comparison we find
\begin{equation*}
  xP_n(x) = (1+x^2)T_n(x).
\end{equation*}
On the other hand, differentiating with respect to $x$ (resp.~$z$) we find
$$
x \frac{\partial}{\partial x} P(x,z) = \frac{\partial}{\partial z} T(x,z)
.
$$
Comparing coefficients we have
$$
xP_n'(x) = T_{n+1}(x);
$$
evaluating   \eqref{eq:derivative:genfunc} at
$x=0$ yields the initial value $P_n(0)=T_n$ and we recover
the explicit formula
\begin{equation}
  \label{eq:Pnx=sumTnkxk}
  P_n(x) = T_n +\sum_{k=1}^{n+1}\frac{T_{n+1}^{(k)}}{k} x^k
  ,
\end{equation}
cf.~\cite[Theorem~1]{Cvijovic:2009:derivative} and \cite{ChangHa:2009}.

On the other hand let us denote by $A_n^{(k)}$ the
\emph{arctangent numbers} 
(see \cite[p.~260]{Comtet} or 
\cite{Cvijovic:2011:higher}) 
defined by their exponential generating function
\begin{equation}
  \label{eq:def-arctannumbers}
  \frac{(\arctan z)^k}{k!} = \sum_{n=k}^\infty \frac{A_n^{(k)}}{n!}z^n;
\end{equation}
notice that $A_n^{(k)}=0$ unless $n-k$ is even
and that up to sign these are the same as the coefficients
of the hyperbolic arctangent function
\begin{equation*}
  \frac{(\atanh z)^k}{k!} = \sum_{n=k}^\infty \frac{\tilde{A}_n^{(k)}}{n!}z^n
  .
\end{equation*}
The latter are nonnegative and 
\begin{equation}
  \label{eq:signarctannumbers}
 A_n^{(k)} = (-i)^ki^n \tilde{A}_n^{(k)}.
\end{equation}

\subsection{Derivatives of $\arctan$}
The derivatives of $\arctan z$ are rational functions and
it is easy to verify by induction that they are given by the following
formulas 
$$
\frac{d}{dz}\arctan z = \frac{1}{1+z^2} = \frac{1}{2i}
\left(
  \frac{1}{z-i} - \frac{1}{z+i}
\right),
$$
and thus
\begin{equation} 
\label{eq:diffarctan}
\frac{d^m}{dz^m}\arctan z= \frac{i(-1)^m(m-1)!}{2}\left((z-i)^{-m}-(z+i)^{-m}\right).
\end{equation}

\subsection{Fa\`a di Bruno's formula}

In this section we briefly recall the combinatorics behind the
composition of exponential generating functions.
We prefer Rota's approach via
the  incidence algebra of the set partition lattice
for the conceptual proofs and concise formulas it provides;
for details we refer the reader to \cite{Aigner:1979}
or the original paper \cite{DoubiletRotaStanley:foundationsVI}.

We denote by $\SP(n)$ the lattice of partitions of the set $\{1,2,\dots,n\}$
under refinement order.
The number of classes (or blocks) a partition $\nu\in\SP(n)$ is 
called its size and denoted by $\abs{\nu}$.

Then Fa\`a di Bruno's formula can be interpreted as an isomorphism
between the reduced incidence algebra of the partition lattices
and  exponential formal power series as follows.
Let $(a_n)_{n\geq 1}$ and $(b_n)_{n\geq 1}$ be sequences and define 
a new sequence by the formula
\begin{equation}
  \label{eq:faadibrunoconvolution}
c_n = \sum_{\nu\in\SP(n)}a_{\abs{\nu}}\prod_{B\in\nu} b_{\abs{B}}
\end{equation}
Then Fa\`a di Bruno's formula
(see \cite[Theorem~5.1.4]{StanleyVol2} or 
\cite[Proposition~5.9]{Aigner:1979})
asserts that  their exponential generating functions
$F_a(z) = \sum_{k=1}^\infty \frac{a_k}{k!}z^k$ 
and $F_b(z) = \sum_{k=1}^\infty \frac{b_k}{k!}z^k$
satisfy the relation
\begin{equation}
  \label{eq:faadibrunocomposition}
  F_c(z) = F_b(F_a(z))
  .  
\end{equation}

Equivalently, given smooth functions $f$ and $g$,
the $m$-th derivative of the composed function is
\begin{equation}
  \label{eq:faadibruno}
    \frac{d^m}{dz^m}f(g(z))
  = \sum_{\nu\in \SP(m)}
  f^{(\abs{\nu})}(g(z))
  \prod_{B\in\nu}
  g^{(\abs{B})}(z) .
\end{equation}

We single out two important functions, namely the $\zeta$-function
$$
\zeta(\nu,\rho)=1,
$$
which corresponds to the sequence $(1,1,\dots)$ and has generating functions
$e^z-1$, and the Möbius function, which is its inverse under convolution, 
and corresponds to the generating function $\log(1+z)$.
In the forthcoming calculations only the values 
\begin{equation}
  \label{eq:moebius}
\mu(\hat{0}_n,\nu)=\prod_{B\in\nu}(-1)^{\abs{B}-1}(\abs{B}-1)!
\end{equation}
will be needed, see   \cite[Example 3.10.4]{StanleyVol1}.
We shall see that they appear when the derivatives
\eqref{eq:diffarctan} are inserted into Fa\`a die Bruno's formula
\eqref{eq:faadibruno}.

\section{Trace formula}

In this section we apply the trace method to the matrix
\eqref{eq:Cn} in order to prove certain properties
of the sum \eqref{eq:cotsum}.
The positivity of the coefficients in the expansion
of the characteristic polynomial could be seen
as a very special case the BMV conjecture \cite{Stahl2013}: if $A$
and $B$ are positive semi-definite matrices, then for all positive
integers $m$, the polynomial in t, $\Tr(A+tB)^m$, has only
non-negative coefficients. The proof below shows that the assertion
is also true whenever $A$ is an orthogonal projection of rank one and
$B$ is a positive or antisymmetric self-adjoint matrix.

\begin{theo}  \label{th:traceformula}
  \begin{enumerate}[(i)]
   \item 
 The cotangent sum \eqref{eq:cotsum}
 can be expressed as
 \begin{equation}
   \label{eq:traceformula}
   \cotsum(m,n,\alpha)=  
      \Tr
   \left(
     (\cot\alpha\J+B_n)^m
   \right)
 \end{equation}
\item 
 There are universal integer valued polynomials   $p_{m,m-2k}(x)$
 with rational coefficients such that
 the cotangent sum \eqref{eq:cotsum}
 can be expressed as a polynomial of degree $m$ in $\cot\alpha$
\begin{equation} 
\label{eq:polynomialcoeff}
 \cotsum(m,n,\alpha) =\sum_{0 \leq k\leq \lfloor m/2\rfloor }p_{m,m-2k}(n)\cot^{m-2k}\alpha.
                          \end{equation}

Moreover, for any $n\in\IN$, the coefficients $p_{m,m-2k}(n)$ are positive integers.
  \end{enumerate}
\end{theo}
\begin{exam}
For example, we have\footnote{We note in passing that there is a misprint in the formula for $S_5(q;\xi)$ 
in \cite[p.~154]{CvijovicKlinowski:2000}.}
\begin{align*}
 \cotsum(1,n,\alpha)&=n\cot \alpha\\
 \cotsum(2,n,\alpha)&=n^2\cot^2 \alpha +n^2-n\\
 \cotsum(3,n,\alpha)&=n^3\cot^3 \alpha+(n^3-n)\cot \alpha\\
 \cotsum(4,n,\alpha)&=n^4\cot^4 \alpha+\frac{4}{3}(n^4-n^2)\cot^2 \alpha+\frac{1}{3}(n^4-4n^2)+n\\
 \cotsum(5,n,\alpha)&=n^5\cot^5 \alpha+\frac{5}{3}(n^5-n^3)\cot^3 \alpha+
 \frac{1}{3}
 (   2n^5-5n^3+n )\cot \alpha
\end{align*}
It will be apparent from \eqref{eq:Smna=sumtanpoly} later that indeed
$\cotsum(m,n,\alpha)$ is a rational polynomial of degree $m$ in both $n$ and
$\cot\alpha$;  explicit expressions for the coefficients are computed
in Corollary~\ref{cor:pmrn}.
\end{exam}

\begin{proof}
  It is clear that the trace
  \eqref{eq:traceformula} is a polynomial of degree at most $n$ in $\cot\alpha$.
  Moreover since the entries of the matrices $J_n$ and $B_n$ are
  integers, the coefficients $p_{m,m-2k}(n)$ are integers as well.
  For positivity, we show that the mixed moments of $J_n$ and $B_n$ are positive.
  To see this, note that $P_n=\frac{1}{n}\J$ is a self-adjoint projection
  of rank $1$. It follows that for any matrix $C_n$ the compression $P_nC_nP_n$
  lies in the 1-dimensional algebra generated by $P_n$, more precisely,
  $P_nC_nP_n= \xi^T C_n\xi P_n$ where $\xi = \frac{1}{\sqrt{n}}(1,1,\dots,1)^T$
  spans the image of $P_n$.
  For our matrix $B_n$ clearly $\xi^T B_n\xi=\sum b_{ij} = 0$ and 
  by antisymmetry, also for   odd powers $\xi^TB_n^k\xi=(-1)^k\xi^TB_n^k\xi=0$.
  It follows that any mixed moment
  \begin{align*}
   \Tr(J_n^{k_1} B_n^{l_1}J_n^{k_2} B_n^{l_2}\dotsm J_n^{k_r}B_n^{l_r})    
    &= n^{k_1+\dots+k_r}\Tr(P_n^{k_1} B_n^{l_1}P_n^{k_2} B_n^{l_2}\dotsm P_n^{k_r}B_n^{l_r}) \\       
    &=n^{k_1+\dots+k_r}\Tr(P_n B_n^{l_1}P_n B_n^{l_2}P_n\dotsm P_n B_n^{l_r}P_n) \\
    &= n^{k_1+\dots+k_r}\xi^T B_n^{l_1}\xi \xi^T B_n^{l_2}\xi \dotsm \xi^T B_n^{l_r}\xi \\
    &=
    \begin{cases}
      = 0 & \text{if some $l_j$ is odd}\\
      > 0 & \text{if all $l_j$ are even}.
    \end{cases}
  \end{align*}
\end{proof}

In particular, $\cotsum(m,n,\alpha)$
evaluates to an integer (natural number) whenever $\cot \alpha$ is an integer (natural number). 
It was observed in \cite{ByrneSmith:1997:integer} 
to the surprise of the authors that
the sums in the next corollary are integer valued;
explicit formulas are computed in Corollary~\ref{cor:cotsumpi/24} below.
We will see later that even for noninteger values of $\cot\alpha$
the sum may evaluate to an integer, e.g.,
for $n=2$ and $\cot\alpha=\frac{1}{2}$, Lucas numbers appear, see
\eqref{eq:Lucas} below.
\begin{cor}
The sums
\begin{align*}
\cotsum(2m-1,n,\pi/4)
&=  \sum_{k=1}^n(-1)^{k}\cot^{2m-1}\frac{(2k-1)\pi}{4n} 
\\
\cotsum(2m,n,\pi/4) 
&=  \sum_{k=1}^n\cot^{2m}\frac{(2k-1)\pi}{4n}
\end{align*}
can be represented as integer-valued polynomials in
$n$ of degrees $2m - 1$ and $2m$, respectively.
\end{cor}
\begin{proof}
Applying Lemma~\ref{lemm:spectrum} to the matrix $\left[\begin{smallmatrix}
    1 &1+i \\
    1-i& 1
    \end{smallmatrix}\right]_n$, we obtain its eigenvalues as
$$\lambda_k=\cot\left(\frac{\pi}{4n}+\frac{k}{n}\pi\right), \text{ for }
  k\in\{1,\dots,n\},$$ because $\alpha=\arccot(1)=\frac{\pi}{4}$.
Let us show how these are related the sums considered by Byrne and Smith~\cite{ByrneSmith:1997:integer}.
Indeed the corresponding power sums are
\begin{align*}
\sum_{k=1}^n\cot^r\left(\frac{\pi}{4n}+\frac{k}{n}\pi\right)&=\sum_{k=1}^{\lfloor n/2\rfloor }\cot^r\left(\frac{\pi}{4n}+\frac{k}{n}\pi\right)+\sum_{k=\lfloor n/2\rfloor+1}^{n}\cot^r\left(\frac{\pi}{4n}+\frac{k}{n}\pi\right)
\intertext{and substituting $\cot(\frac{\pi}{4n}+\frac{k}{n}\pi)=-\cot(-\frac{\pi}{4n}+\frac{n-k}{n}\pi)$ into the second sum, we get}
&=\sum_{k=1}^{\lfloor
  n/2\rfloor}\cot^r\left(\frac{\pi}{4n}+\frac{k}{n}\pi\right)+\sum_{k=0}^{n-\lfloor
  n/2\rfloor-1}(-1)^r\cot^r\left(-\frac{\pi}{4n}+\frac{k}{n}\pi\right)\\
&=
  \begin{cases}
    -\sum_{k=1}^n(-1)^{k-1}\cot^{r}\frac{(2k-1)\pi}{4n} & \text{if $r$ is odd,}\\
    \sum_{k=1}^n\cot^{r}\frac{(2k-1)\pi}{4n} & \text{if $r$ is even.}
  \end{cases}
\end{align*}
\end{proof}

\section{Generating functions}
\label{sec:genfunc}
In the present section we compute the generating function
of the cotangent sums \eqref{eq:cotsum}, for fixed $n$, i.e.,  
\begin{equation*}
  F_n(z,\alpha) = \sum_{m=0}^\infty\cotsum(m,n,\alpha)z^m,
\end{equation*}
which is the moment generating function
of the matrix $\cot\alpha \J+B_n$ with respect to the nonnormalized trace.
Moreover we will compute the moment generating function of the matrix $B_n$
with respect to the nonnormalized trace and with respect to the state
$\omega$
with density matrix $P_n=\frac{1}{n}\J$, that is, 
\begin{equation*}
\omega(C)=\Tr(P_nC)=\frac{1}{n}\sum_{i,j}^n c_{ij}=\xi^TC\xi 
\end{equation*}
where as above by $\xi$ we denote the unit vector
$\xi=\frac{1}{\sqrt{n}}(1,1,\dots,1)^T$ and $C=[c_{i,j}]_{i,j=1}^n\in M_n(\C)$.
The moment generating functions 
\begin{align*}
M_{x\J+B_n}(z) &= \Tr((I-z(x\J+B_n))^{-1}),\\
M_{B_n}(z) &= \Tr((I-zB_n)^{-1}),
\intertext{
with respect to
the trace are easy to compute directly through the characteristic polynomials.
On the other hand, direct computation of}
\M_{B_n}(z) &= \omega ((I-zB_n)^{-1}) = \Tr(P_n(I-zB_n)^{-1})
\end{align*}
requires information about the eigenvectors which we could not obtain.
It will therefore be computed indirectly. 
The tangent function and its inverse will play a major role in these
computations and we collect some facts about these functions first.

\subsection{Generating function for cotangent sums}
\begin{prop}
  For fixed $n$ the ordinary generating function of the cotangent sums
  \eqref{eq:cotsum} is
  \begin{equation}
    F_n(z,\alpha) =\sum_{m=0}^\infty \sum_{k=0}^{n-1}\cot^m\frac{\alpha+k\pi}{n}\,z^m 
    = \frac{n}{1+z^2}(1-z\cot(n\arctan z-\alpha))      
    \label{eq:genfunc}
  \end{equation}
  More generally, the moment generating function of the
  matrix pencil $xJ_n+B_n$ is
  \begin{equation}   \label{eq:MxJn+Bn}
  M_{xJ_n+B_n}(z) 
  = \frac{n}{1+z^2}
  \left(
    1 + z\frac{x+\tan(n\arctan z)}{1-x\tan(n\arctan z)}
  \right).
  \end{equation}
\end{prop}
\begin{proof}
Let $\theta_k=\frac{\alpha+k\pi}{n}$ and recall that
we have realized $\cot\theta_k$ as roots of the polynomial 
\eqref{eq:pnlambda}.
Thus we can write the generating function of the sequence \eqref{eq:cotsum}
as the logarithmic derivative of this polynomial. Indeed,
\begin{align*}
  g_n(z) 
  &= \sum_{k=0}^{n-1} \frac{1}{z-\cot\theta_k}\\
  &= \frac{\chi_n'(\alpha;z)}{\chi_n(\alpha;z)}   \\
  &= n\frac{
    (\cot\alpha+i)(z-i)^{n-1}
    -
    (\cot\alpha-i)(z+i)^{n-1}
  }{
    (\cot\alpha+i)(z-i)^n
    -
    (\cot\alpha-i)(z+i)^n
  }
\end{align*}
then the ordinary generating function is
\begin{align*}
 F_n(z,\alpha) 
 &= \frac{1}{z}g_n\left(\frac{1}{z}\right)  \\
  &= n\frac{
    (\cot\alpha+i)(1-iz)^{n-1}
    -
    (\cot\alpha-i)(1+iz)^{n-1}
  }{
    (\cot\alpha+i)(1-iz)^n
    -
    (\cot\alpha-i)(1+iz)^n
    }
  \\
 &= \frac{n}{1+z^2}
   \frac{
    (\cot\alpha+i)(1-iz)^n(1+iz)
    -
    (\cot\alpha-i)(1+iz)^n(1-iz)
  }{
    (\cot\alpha+i)(1-iz)^n
    -
    (\cot\alpha-i)(1+iz)^n
    }
  \\
 &= \frac{n}{1+z^2}
   \left(
   1+iz
   \frac{
    (\cot\alpha+i)(1-iz)^n
    +
    (\cot\alpha-i)(1+iz)^n
  }{
    (\cot\alpha+i)(1-iz)^n
    -
    (\cot\alpha-i)(1+iz)^n
    }
   \right)
  \\
 &= \frac{n}{1+z^2}
   \left(
   1-z\cot(n\arctan z-\alpha)
   \right)
\end{align*}
where in the last step we used identity
  \eqref{eq:cotnx-alpha}.
  The general formula  \eqref{eq:MxJn+Bn} follows by substituting
  $\alpha=\arccot x$ and the addition formula for tangent  \eqref{eq:tanx+y}.
\end{proof}

\begin{Rem}
  In the cases $\alpha=0$ ($\alpha=\pi/2$, resp.) formula~\eqref{eq:genfunc} reproduces
  \cite[Formula (A7.2) (resp.~ (C6.2))]{ChuMarini:1999:partial}.
  At a first glance for $\alpha=0$  the sum diverges:
  $\sum_{k=0}^{n-1}\cot^m \frac{k\pi}{n}=\pm \infty$.
  However \cite[Formula (A7.1)]{ChuMarini:1999:partial} the sum starts at
  $k=1$, i.e.,
  $\sum_{k=1}^{n-1}\cot^m \frac{k\pi}{n}$.
  Inspection of the partial fraction expansion of the generating 
  function~\eqref{eq:genfunc} however 
  reveals that the term $ \frac{1}{1-z\cot\theta_0}$ vanishes as $\alpha$ tends
  to zero and the generating function becomes
          $$F_n(z,0)=\sum_{k=1}^{n-1} \frac{1}{1-z\cot\theta_k}$$
  and this is indeed the generating function of the sums  $\sum_{k=1}^{n-1}\cot^m \frac{k\pi}{n} $. 
    In the case $\alpha=\pi/2$ formula~\eqref{eq:genfunc} reproduces
  \cite[Formula (C6.2)]{ChuMarini:1999:partial}.
  Indeed, since $\cot(\alpha-\pi/2)=-\tan\alpha$ we have
  $
  M_{B_n}(z)=
    \Tr((I-zB_n)^{-1}) = \frac{n(1+z\tan(n\arctan z))}{1+z^2}
   $.
\end{Rem}

\subsection{A functional relation}
In this section we indicate an algorithm to calculate the coefficients $p_{m,m-2k}(n)$,  
which is the main contribution of this paper. 
The following lemma is a special case of cyclic Boolean convolution
\cite{ArizmendiHasebeLehner:2019:cyclic}; 
we reproduce the calculation here for the reader's convenience.
\begin{lemm}
  The generating functions $F_n(z,\alpha)$ and $\M_B(z)$ satisfy the
  relation
  \begin{equation*}
   M_{xJ_n+B_n}(z)=
    \frac{nxz\frac{d}{dz} z\M_{B_n}(z)}{1-nxz\M_{B_n}(z)} +  M_{B_n}(z)
  \end{equation*}
\end{lemm}
\begin{proof}
  The first terms of the power series are easy to calculate
  \begin{equation}
    \label{eq:MxJnBn.1}
    M_{xJ_n+B_n}(z)=  n + xnz + \sum_{m\geq 2}     \Tr((x\J+B_n)^m) z^m
  \end{equation}
  and  for $m\geq2$ we expand the powers and arrange the resulting words
  according to the last letter:
  \begin{align*}
    \Tr((&x\J+B_n)^m) 
    = \Tr\Bigl( (x\J)^m +B_n^m   \\
    & \phantom{==} + \sum_{\substack{
           k\geq1\\
           p_0\geq 0\\
           p_1,p_2,\dots,p_k\geq 1\\
           q_1,q_2,\dots,q_k\geq 1\\
           p_0+q_1+p_1+\dots+q_k+p_k=m
         }
       }
       B_n^{p_0}(x\J)^{q_1} B_n^{p_1}(x\J)^{q_2} B_n^{p_2}\dotsm (x\J)^{q_k} B_n^{p_k}
       \\
    & \phantom{==} + \sum_{\substack{
           k\geq1\\
           q_0\geq 0\\
           p_1,p_2,\dots,p_k\geq 1\\
           q_1,q_2,\dots,q_k\geq 1\\
           q_0+p_1+q_1+\dots+p_k+q_k=m
         }
       }
       (x\J)^{q_0} B_n^{p_1}(x\J)^{q_1} B_n^{p_2}(xJ)^{q_2}\dotsm B_n^{p_k} (x\J)^{q_k} 
       \Bigr)
       \\
    &= \Tr( B_n^m) + \Tr((x\J)^m) \\
    & \phantom{==} + \sum_{\substack{
           k\geq1\\
           p_0\geq 0\\
           p_1,p_2,\dots,p_k\geq 1\\
           q_1,q_2,\dots,q_k\geq 1\\
           p_0+q_1+p_1+\dots+q_k+p_k=m
         }
       }
       (xn)^{q_1+q_2+\dots+q_k}
       \Tr(PB_n^{p_1})\Tr(PB_n^{p_2})\dotsm \Tr(PB_n^{p_{k-1}})\Tr(PB_n^{p_k+p_0})
       \\
    & \phantom{==} + \sum_{\substack{
           k\geq1\\
           q_0\geq 0\\
           p_1,p_2,\dots,p_k\geq 1\\
           q_1,q_2,\dots,q_k\geq 1\\
           q_0+p_1+q_1+\dots+p_k+q_k=m
         }
       }
       (xn)^{q_0+q_1+\dots+q_k}
       \Tr(PB_n^{p_1}) \Tr(PB_n^{p_2})\dotsm \Tr(PB_n^{p_k})
  \end{align*}

  Inserting this expansion into \eqref{eq:MxJnBn.1} we obtain
  \begin{align*}
   M_{xJ_n+B_n}(z)
 &= n + nxz + \sum_{m\geq 2} \Tr(B_n^m)z^m + \sum_{m\geq2}(nxz)^m\\
    & \phantom{==}
      + \sum_{k\geq1}
      \left(
        \frac{xnz}{1-xnz}
      \right)^k
      (\M_{B_n}(z)-1)^{k-1} \hat{M}_{B_n}(z)
       \\
    & \phantom{==} 
      +
      \frac{1}{1-xnz}      
      \sum_{k\geq1}
      \left(
        \frac{xnz}{1-xnz}
      \right)^k
        (\M_{B_n}(z)-1)^k
        \\
    &= \Tr((I-zB_n)^{-1}) + \frac{nxz}{1-nxz}\\
    &\phantom{==}
       + \frac{nxz}{1-nxz}\frac{1}{1-\frac{nxz}{1-nxz}(\M_{B_n}(z)-1)}\hat{M}_{B_n}(z)
       + \frac{1}{1-nxz}
       \left(
         \frac{1}{1-\frac{nxz}{1-nxz}(\M_{B_n}(z)-1)}
         -1
       \right)\\
    &= M_{B_n}(z) + \frac{nxz}{1-nxz}
       + \frac{nxz\hat{M}_{B_n}(z)}{1-nxz\M_{B_n}(z)}
       + \frac{1}{1-nxz\M_{B_n}(z)}
       - \frac{1}{1-nxz}
       \\
    &= M_{B_n}(z)
       + \frac{1+nxz\hat{M}_{B_n}(z)}{1-nxz\M_{B_n}(z)} -1
       \\
    &= M_{B_n}(z)
       + \frac{nxz(\hat{M}_{B_n}(z)+\M_{B_n}(z)) }{1-nxz\M_{B_n}(z)}
  \end{align*}
  where
  \begin{align*}
    \hat{M}_{B_n}(z) 
    &= \sum_{p_0\geq 0, p\geq 1} \Tr(PB_n^{p_0+p})z^{p_0+p}\\
    &= \sum_{m=1}^\infty \sum_{\substack{p_0\geq 0\\ p\geq 1\\ p_0+p=m}} \Tr(PB_n^m)z^m\\
    &= \sum_{m=1}^\infty m \Tr(PB_n^m)z^m\\
    &= z\frac{d}{dz}\M_{B_n}(z)
  \end{align*}
\end{proof}

\begin{lemm} \label{rownaiefunckyjne}
  For any $x$, the differential equation 
  \begin{equation*}
    \frac{nxz g'(z)}{1-nxg(z)} + M_{B_n}(z) = M_{xJ_n+B_n}(z)
  \end{equation*}
  with initial condition $g(0)=1$ 
  has the unique solution $g(z)=z\M_{B_n}(z)=\frac1n\tan(n\arctan z)$.
\end{lemm}
\begin{proof}
  Observe that the  considered expression 
  can be rewritten as a first order linear equation in standard form 
  $$g'(z)+q(z)g(z)=p(z).$$
  which has a unique solution.
  and direct verification shows that it is given by
  $g(z)=\frac{\tan(n\arctan z)}{n}$.
\end{proof}

\section{Combinatorial interpretation}
In this section we indicate explicit combinatorial interpretations
 of the coefficients of 
polynomials \eqref{eq:polynomialcoeff} which express the value of trace of matrices in terms of Dyck paths and rooted binary  trees. We emphasize that these coefficients $p_{m,k}(n)$ are nonzero, whenever $m$ and $k$ have the
same parity. 

\subsection{Dimension 2}
First let us record that for $n=2$ at offset  $\cot\alpha_0=\frac{1}{2}$ we recover the well-known sequence Lucas numbers (A000032 in the On-Line Encyclopedia of Integer Sequences \cite{Sloane}).
Indeed, the characteristic polynomial   \eqref{eq:pnlambda} is
$$
\chi_2(\lambda) = \Im (\frac{1}{2}+i)(\lambda-i)^2 = \lambda^2-\lambda-1
$$
and the roots are the golden ratios $\phi_{\pm}=\frac{1\pm\sqrt{5}}{2}$ 
with moments 
\begin{equation}
\label{eq:Lucas} 
\cotsum(m,2,\alpha_0) = L_m = \phi_+^m + \phi_-^m
\end{equation}
satisfying the recurrence relation
$$
L_m =
\begin{cases}
  2 & m=0\\
  1 & m=1\\
  L_{m-1}+L_{m-2} &m\geq2
\end{cases}
.
$$

\subsection{Interpretation of $\Tr(\J B_n^{2m})$ in terms of Dyck paths}
For the general case we establish some recurrence relations.
An explicit formula will be established in Corollary~\ref{cor:TrJBmexpl}
below.
\begin{prop}   The moments
  \begin{equation}
    \label{eq:def-dnm}
    d_{n,m}=\Tr(\J B_n^{2m})
  \end{equation}
satisfy the recurrence    
\begin{equation} \label{recursionrelation}
\begin{aligned}
d_{n,0}&=1+d_{n-1,0}=n ,\\ 
d_{1,m}&=\delta_{0,m},\\
d_{n,m}&=d_{n-1,m}+\sum_{k=0}^{m-1}d_{n-1,k}d_{n,m-k-1}.
\end{aligned}
\end{equation}
which is reminiscent of the  recurrence relations for the Motzkin numbers. 
\end{prop}
\begin{proof}
The function $Q_n(z)=\frac{\tan(n\arctan z)}{z}
$
is rational by \eqref{eq:tannatanz=pn}.
Indeed  $Q_1(z)=1$ and for higher order  the addition
theorem for the tangent function \eqref{eq:tanx+y} yields
the recurrence
  \begin{align*}
Q_n(z)&=\frac{1}{z}\tan(\arctan z+(n-1)\arctan z)=\frac{z+\tan((n-1)\arctan z)}{z-z^2\tan((n-1)\arctan z)}=\frac{1+Q_{n-1}(z)}{1-z^2Q_{n-1}(z)}
\intertext{or eqivalently} 
Q_n(z)&=1+Q_{n-1}(z)+z^2Q_{n-1}(z)Q_{n}(z).
\end{align*} 
  From Lemma~\ref{rownaiefunckyjne} we infer that $\sum_{m=0}^\infty d_{n,m}z^{2m}=Q_n(z)$ and we can readily calculate the required recurrence for
the moments $d_{n,m}$.
\end{proof}

The continued fraction of the rational function $Q_n(z)$ is finite
and was computed in \cite {OliverProdinger:2012:continued}:
$$
Q_n(z) = \cfrac{n}{1-\cfrac{\frac{(n+1)(n-1)}                                 {1\cdot3}z^2}                           {1-\cfrac{\frac{(n+2)(n-2)}{3\cdot5}z^2}                                    {1-\cfrac{\frac{(n+3)(n-3)}{5\cdot7}z^2}{\ddots}}}}
$$
We can thus infer from Flajolet's theory of continued fractions~\cite{Flajolet1980} the following formula for the moments $d_{n,m}$.
\begin{theo}
  \begin{equation*}
    d_{n,m} = n \sum_{\pi\in\Dyck_m} w(\pi)
  \end{equation*}
  where the sum runs over Dyck paths of length at most $2m$ with weights
  $a_{k-1}=\frac{n-k}{2k-1}$,   $b_k=\frac{n+k}{2k+1}$, $k=1,2,\dots,n$.
\end{theo}
\begin{exam}
  For $n=3$ the generating function is
  $$
  Q_3(z) = \frac{{{z}^{2}} -3} {{3 \, {{z}^{2}}} -1} 
  =
  3+{8 \, {{x}^{2}}}+{{24} \, {{x}^{4}}}+{{72} \, {{x}^{6}}}+{{216} \, {{x}^{8}}}+{{648} \, {{x}^{{10}}}}+{O \left({{{x}^{{11}}}} \right)}
  $$
  and indeed for $n=3$ with weights
  $$
  a_0= \frac{2}{1}, 
  \quad
  a_1=\frac{1}{3},
  \quad
  b_1=\frac{4}{3},
  \quad
  b_2=\frac{5}{5}
  $$
  we have
  \tikzset{every picture/.style={scale=0.4}}

  \begin{align*} 
    d_{3,1} &=
    3\cdot
    \left(
      \dyckfrac{\makeatletter{}\begin{tikzpicture}[trim left,scale=1.2]
\node[draw,circle,inner sep=1pt,fill] at (0,0) {};
\node[draw,circle,inner sep=1pt,fill] at (1,1) {};
\node[draw,circle,inner sep=1pt,fill] at (2,0) {};
\draw (0,0)--(1,1)--(2,0);
\end{tikzpicture}
 }
               {\frac{2}{1}\cdot\frac{4}{3}}
    \right)
    = 8
\\
    d_{3,2} &=
    3\cdot
    \left(
      \dyckfrac{\makeatletter{}\begin{tikzpicture}[trim left,scale=1.3]
\node[draw,circle,inner sep=1pt,fill] at (0,0) {};
\node[draw,circle,inner sep=1pt,fill] at (1,1) {};
\node[draw,circle,inner sep=1pt,fill] at (2,0) {};
\node[draw,circle,inner sep=1pt,fill] at (3,1) {};
\node[draw,circle,inner sep=1pt,fill] at (4,0) {};
\draw (0,0)--(1,1)--(2,0)--(3,1)--(4,0);
\end{tikzpicture}
 }
               {\frac{2}{1}\cdot\frac{4}{3}\cdot\frac{2}{1}\cdot\frac{4}{3}}
    +
    \dyckfrac{\makeatletter{}\begin{tikzpicture}[trim left,scale=1.4]
\node[draw,circle,inner sep=1pt,fill] at (0,0) {};
\node[draw,circle,inner sep=1pt,fill] at (1,1) {};
\node[draw,circle,inner sep=1pt,fill] at (2,2) {};
\node[draw,circle,inner sep=1pt,fill] at (3,1) {};
\node[draw,circle,inner sep=1pt,fill] at (4,0) {};
\draw (0,0)--(1,1)--(2,2)--(3,1)--(4,0);
\end{tikzpicture}
 }
             {\frac{2}{1}\cdot\frac{1}{3}\cdot\frac{5}{5}\cdot\frac{4}{3}}
    \right)
    \\
    &= 3\cdot
    \left(
      \frac{64}{9} + \frac{8}{9}  
    \right)
    = 24
\\
    d_{3,3} &=
    3\cdot
    \Biggl(
      \dyckfrac{\makeatletter{}\begin{tikzpicture}[trim left,scale=1.4]
\node[draw,circle,inner sep=1pt,fill] at (0,0) {};
\node[draw,circle,inner sep=1pt,fill] at (1,1) {};
\node[draw,circle,inner sep=1pt,fill] at (2,2) {};
\node[draw,circle,inner sep=1pt,fill] at (3,3) {};
\node[draw,circle,inner sep=1pt,fill] at (4,2) {};
\node[draw,circle,inner sep=1pt,fill] at (5,1) {};
\node[draw,circle,inner sep=1pt,fill] at (6,0) {};
\draw (0,0)--(1,1)--(2,2)--(3,3)--(4,2)--(5,1)--(6,0);
\end{tikzpicture}
 }
               {      0}
    +
    \dyckfrac{\makeatletter{}\begin{tikzpicture}[trim left,scale=1.4]
\node[draw,circle,inner sep=1pt,fill] at (0,0) {};
\node[draw,circle,inner sep=1pt,fill] at (1,1) {};
\node[draw,circle,inner sep=1pt,fill] at (2,2) {};
\node[draw,circle,inner sep=1pt,fill] at (3,1) {};
\node[draw,circle,inner sep=1pt,fill] at (4,2) {};
\node[draw,circle,inner sep=1pt,fill] at (5,1) {};
\node[draw,circle,inner sep=1pt,fill] at (6,0) {};
\draw (0,0)--(1,1)--(2,2)--(3,1)--(4,2)--(5,1)--(6,0);
\end{tikzpicture}
 }
             {\frac{2}{1}\cdot\frac{1}{3}\cdot\frac{5}{5}\cdot\frac{1}{3}\cdot\frac{5}{5}\cdot\frac{4}{3}}
    +
    \dyckfrac{\makeatletter{}\begin{tikzpicture}[trim left,scale=1.4]
\node[draw,circle,inner sep=1pt,fill] at (0,0) {};
\node[draw,circle,inner sep=1pt,fill] at (1,1) {};
\node[draw,circle,inner sep=1pt,fill] at (2,0) {};
\node[draw,circle,inner sep=1pt,fill] at (3,1) {};
\node[draw,circle,inner sep=1pt,fill] at (4,2) {};
\node[draw,circle,inner sep=1pt,fill] at (5,1) {};
\node[draw,circle,inner sep=1pt,fill] at (6,0) {};
\draw (0,0)--(1,1)--(2,0)--(3,1)--(4,2)--(5,1)--(6,0);
\end{tikzpicture}
 }
              {\frac{2}{1}\cdot\frac{4}{3}\cdot\frac{2}{1}\cdot\frac{1}{3}\cdot\frac{5}{5}\cdot\frac{4}{3}}
    \\
    &\phantom{=====}
    +
    \dyckfrac{\makeatletter{}\begin{tikzpicture}[trim left,scale=1.4]
\node[draw,circle,inner sep=1pt,fill] at (0,0) {};
\node[draw,circle,inner sep=1pt,fill] at (1,1) {};
\node[draw,circle,inner sep=1pt,fill] at (2,2) {};
\node[draw,circle,inner sep=1pt,fill] at (3,1) {};
\node[draw,circle,inner sep=1pt,fill] at (4,0) {};
\node[draw,circle,inner sep=1pt,fill] at (5,1) {};
\node[draw,circle,inner sep=1pt,fill] at (6,0) {};
\draw (0,0)--(1,1)--(2,2)--(3,1)--(4,0)--(5,1)--(6,0);
\end{tikzpicture}
 }
             {\frac{2}{1}\cdot\frac{1}{3}\cdot\frac{5}{5}\cdot\frac{4}{3}\cdot\frac{2}{1}\cdot\frac{4}{3}}
    +
    \dyckfrac{\makeatletter{}\begin{tikzpicture}[trim left,scale=1.4]
\node[draw,circle,inner sep=1pt,fill] at (0,0) {};
\node[draw,circle,inner sep=1pt,fill] at (1,1) {};
\node[draw,circle,inner sep=1pt,fill] at (2,0) {};
\node[draw,circle,inner sep=1pt,fill] at (3,1) {};
\node[draw,circle,inner sep=1pt,fill] at (4,0) {};
\node[draw,circle,inner sep=1pt,fill] at (5,1) {};
\node[draw,circle,inner sep=1pt,fill] at (6,0) {};
\draw (0,0)--(1,1)--(2,0)--(3,1)--(4,0)--(5,1)--(6,0);
\end{tikzpicture}
 }
             {\frac{2}{1}\cdot\frac{4}{3}\cdot\frac{2}{1}\cdot\frac{4}{3}\cdot\frac{2}{1}\cdot\frac{4}{3}}
    \Biggr)
    \\
    &= 3\cdot
    \left(
      0 + \frac{8}{27} + \frac{64}{27} + \frac{64}{27} + \frac{512}{27}
    \right)
    = 72
  \end{align*}
  etc.
  \tikzset{every picture/.style={scale=1.0}}
\end{exam}

\subsection{Interpretation of $\Tr(\J B_n^{2m})$ in terms of binary trees}

Set   $e_{n,k}=d_{n,k-1}$ and $d_{n,-1}=1$, then  recursion \eqref{recursionrelation} can be rewritten more compactly as 
\begin{align*} 
e_{n,0}&=1 ,\\ 
e_{n,1}&=n ,\\
e_{1,m}&=\delta_{0,m-1},\\
e_{n,m}&=\sum_{k=1}^{m}e_{n-1,k}e_{n,m-k}.
\end{align*}
which is reminiscent of the Catalan recurrence relations. 

\begin{defi}
A \emph{rooted binary tree} is a
rooted tree in which each node has at most two children,
one of which we distinguish as \emph{firstborn}.
We use the convention that the root is not a child and therefore does
not count as a firstborn; our trees are unordered but we take the convention
that firstborns are always drawn on the right.
We denote by $T_{n,m}$ the set of  rooted binary trees with $m$ leaves,  
such that each leaf has a brother
and every path emanating from the root contains at most $n-1$  firstborns.
We note that 
the  set $T_{n,m}$ is empty unless  $m\leq n.$ 

For a  rooted binary tree $\tau\in T_{n,m}$ 
we  denote by $\Paths(\tau)$ the set of maximal rooted paths.
For such a path $p\in\Paths(\tau)$ we denote by
 $r(p)$ be the number of firstborn nodes occurring in $p$
and its \emph{weight} 
$\omega(p) = n-r(p)$ which is a number between $1$ and $n$.
\end{defi}
\begin{figure}[h]
\begin{tikzpicture}[->,>=stealth',level/.style={sibling distance = 1cm/#1,
  level distance = 1cm}] 
\node [arn_n] {}
     node [arn_n] {} 
            child{ node [arn_n] {} 
            }
            child{ node [arn_n] {}
            }  
            (0,-1.82) node[] (pom1) {$3\times 2$ }                          
; 
\end{tikzpicture}
\begin{tikzpicture}[->,>=stealth',level/.style={sibling distance = 2cm/#1,
  level distance = 1.2cm}] 
\node [arn_n] {}
child {node [arn_n] {} 
    child{ node [arn_n] {}                            
    }
    child{ node [arn_n] {}
		}}
		 (0,-3.2) node[] (pom1) {$2\times 1$ }
; 
\end{tikzpicture}
     	 \caption{$T_{3,2}$ with corresponding weight of paths. }
\label{figure32}
\end{figure}
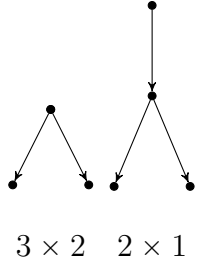
 
\begin{figure}[h]
\begin{tikzpicture}[->,>=stealth',level/.style={sibling distance = 2cm/#1,
  level distance = 1.2cm}] 
\node [arn_n] {}
    child{ node [arn_n] {}                                   
    }
    child{ node [arn_n] {}
    child{ node [arn_n] {} 
            }
            child{ node [arn_n] {}
            }      
		}
		  (0,-3) node[] (pom1) {$3\times 2\times 1$ }
; 
\end{tikzpicture}
\begin{tikzpicture}[->,>=stealth',level/.style={sibling distance = 2cm/#1,
  level distance = 1.2cm}] 
\node [arn_n] {}
    child{ node [arn_n] {} 
            child{ node [arn_n] {} 
            }
            child{ node [arn_n] {}
            }                            
    }
    child{ node [arn_n] {}
		}
			  (0,-3) node[] (pom1) {$3\times 2\times 2$ }
; 
\end{tikzpicture}
\begin{tikzpicture}[->,>=stealth',level/.style={sibling distance = 3cm/#1,
  level distance = 1.2cm}] 
\node [arn_n] {}
    child{ node [arn_n] {} 
    child { node [arn_n] {} 
            child{ node [arn_n] {} 
            }
            child{ node [arn_n] {}
            }                            
    }}
    child{ node [arn_n] {}
		}
			  (0,-4.2) node[] (pom1) {$2\times 1\times 2$ }
; 
\end{tikzpicture}
\begin{tikzpicture}[->,>=stealth',level/.style={sibling distance = 3cm/#1,
  level distance = 1.2cm}] 
\node [arn_n] {}
child {node [arn_n] {} 
    child{ node [arn_n] {} 
            child{ node [arn_n] {} 
            }
            child{ node [arn_n] {}
            }                            
    }
    child{ node [arn_n] {}
		}}
			  (0,-4.2) node[] (pom1) {$2\times 1\times 1$ }
; 
\end{tikzpicture}

                             		 \caption{$T_{3,3}$ with corresponding weight of paths. }
\label{figure33}
\end{figure}

\begin{theo} Let $1\leq m\leq n$, then 
  \begin{align}\label{eq:enm}
    e_{n,m} =  \sum_{\tau\in T_{n,m}} \prod_{p\in \Paths(\tau)}\omega(p).
   \end{align}
    \end{theo}
\begin{proof}
Let us denote the right-hand side of  \eqref{eq:enm} by $c_{n,m}$. 
If $m=1$ the root is the only node and does not count as a firstborn,
therefore $e_{n,1}=c_{n,1}=n$.
Moreover $T_{2,2}$ only contains one tree of weight $e_{2,2}=c_{2,2}=2$.
More generally $T_{n,2}$ contains $(n-1)$ trees and
$e_{n,2}=c_{n,2}=n(n-1)+(n-1)(n-2)+\dots+2\times 1$.  
So, it is sufficient to verify that $e_{n,m}=c_{n,m}$ for $n,m\geq 3.$
Notice that any rooted binary tree can be viewed as one or two
(non-empty because  $n,m\geq 3$) rooted binary trees grafted onto a common root;
see Fig.~\ref{figure32} and \ref{figure33}. 
Thus in order to create all possible  binary trees we start with a root vertex, 
and one child (Case 1) or  two children (Case 2a and 2b) 
with all possible choices of the subtrees trees  as shown in the diagram below.
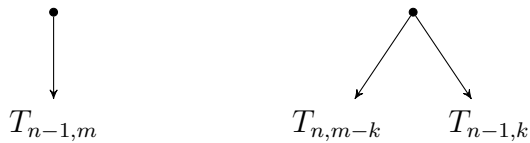
\begin{center}
\begin{tikzpicture}[->,>=stealth',level/.style={sibling distance = 2cm/#1,
  level distance = 1.5cm}] 
\node    [arn_n] {} 
            child{ node   {$T_{n-1,m}$} 
            }                            
; 
\end{tikzpicture}
\hspace{2cm}
\begin{tikzpicture}[->,>=stealth',level/.style={sibling distance = 2cm/#1,
  level distance = 1.5cm}] 
\node [arn_n] {} 
            child{ node {$T_{n,m-k}$} 
            }
            child{ node   {$T_{n-1,k}$}
            }                            
; 
\end{tikzpicture} 
\end{center}
Case 1.  Assume that the root has only one child $v_0$.
Then every path from the root to a leaf with at most $n-1$ firstborns  consists of the first step and a path from $v_0$ with at most $n-2$ firstborns.
So the weight remains the same and 
the number of leaves remains $m$.

Case 2a). Let $\tau$ be such a tree and $p$ a path passing through the firstborn child $v_0$. Then we can
consider the latter as root vertex of new binary tree in  $T_{n-1,k}$ 
  with $k$ leaves for $k\in\{1,\dots, m-1\}$. Denote by $p'$ the restriction of the path $p$ to this
subtree. Observe that $p'$ contains  at most $n-2$  firstborns  because $v_0$ already counts as a firstborn and the weights of $p$ and $p'$ coincide.
Indeed  $r(p)=r(p')+1$ and so $\omega(p)=n-r(p)=n-1-r(p')=\omega(p')$.

Case 2b). Let now $p$ be a path passing through the other child, that is,
$p'$ is a path in a tree from  $T_{n,m-k}$ and again the weight does not change.

Finally we have 
$$\sum_{\tau\in T_{n,m}} \prod_{p\in \Paths(\tau)}\omega(p)=\sum_{k=1}^{m-1}\sum_{\tau\in T_{n,m-k}} \prod_{p\in \Paths(\tau)}\omega(p)\sum_{\tau\in T_{n-1,k}} \prod_{p\in \Paths(\tau)}\omega(p)+\sum_{\tau\in T_{n-1,m}} \prod_{p\in \Paths(\tau)}\omega(p),$$
and now we can write 
$$c_{n,m}=\sum_{k=1}^{m-1}c_{n-1,k}c_{n,m-k}+c_{n-1,m}$$
Thus we see that $c_{n,m}=e_{n,m}.$
\end{proof}
\subsection{Interpretation of $p_{m,k}(n)$ for $k\geq 2$}
The combinatorial objects that we consider now are called  circular binary forests. 
\begin{defi}  Assume that $m,k\in\N$ have the same parity.   
For $k\in\{2,\dots,m\}$  
a \emph{circular  binary forest}   $T^E_{n,m,k}$ of degree $k$  
is a set of $k$ binary trees as above arranged on a circle with a
total number of $m$ leaves,
see Figure~\ref{fig:circularforest} for an example.
\begin{figure}[h]
  \centering
\makeatletter{}\begin{tikzpicture}
\draw[black] (0,0) circle (3 cm);

\foreach \x in {1,73,...,289}
   \draw[dashed](\x:3 cm) --  (0,0);

\coordinate (O) at (0,0);
\fill(O) circle(2pt);

\coordinate (A1) at (24:0.5cm);
\fill(A1) circle(2pt);
\draw(A1) circle(3pt);

\coordinate (A11) at (20:1.5cm);
\fill(A11) circle(2pt);
\draw(A11) circle(3pt);

\coordinate (A111) at (10:2cm);
\fill(A111) circle(2pt);
\draw(A111) circle(3pt);

\coordinate (A1111) at (7:3cm);
\fill(A1111) circle(2pt);
\draw(A1111) circle(3pt);

\coordinate (A1112) at (15:3cm);
\fill(A1112) circle(2pt);

\coordinate (A112) at (25:2.2cm);
\fill(A112) circle(2pt);

\coordinate (A1121) at (20:3cm);
\fill(A1121) circle(2pt);
\draw(A1121) circle(3pt);

\coordinate (A1122) at (28:3cm);
\fill(A1122) circle(2pt);

\coordinate (A12) at (35:1.2cm);
\fill(A12) circle(2pt);

\coordinate (A121) at (40:1.8cm);
\fill(A121) circle(2pt);
\draw(A121) circle(3pt);

\coordinate (A1211) at (35:3cm);
\fill(A1211) circle(2pt);
\draw(A1211) circle(3pt);

\coordinate (A1212) at (44:2.5cm);
\fill(A1212) circle(2pt);

\coordinate (A12121) at (42:3cm);
\fill(A12121) circle(2pt);
\draw(A12121) circle(3pt);

\coordinate (A12122) at (48:3cm);
\fill(A12122) circle(2pt);

\coordinate (A2) at (55:0.8cm);
\fill(A2) circle(2pt);

\coordinate (A21) at (60:1.5cm);
\fill(A21) circle(2pt);
\draw(A21) circle(3pt);

\coordinate (A211) at (60:2.3cm);
\fill(A211) circle(2pt);
\draw(A211) circle(3pt);

\coordinate (A2111) at (55:3cm);
\fill(A2111) circle(2pt);
\draw(A2111) circle(3pt);

\coordinate (A2112) at (62:3cm);
\fill(A2112) circle(2pt);

\coordinate (A212) at (67:3cm);
\fill(A212) circle(2pt);

\draw (O)--(A1);

\draw (A1)--(A11);
  \draw (A11)--(A111);
    \draw (A111)--(A1111);
    \draw (A111)--(A1112);
  \draw (A11)--(A112);
    \draw (A112)--(A1121);
    \draw (A112)--(A1122);
\draw (A1)--(A12);
  \draw (A12)--(A121);
    \draw (A121)--(A1211);
    \draw (A121)--(A1212);
      \draw (A1212)--(A12121);
      \draw (A1212)--(A12122);

\draw (O)--(A2)--(A21)--(A211)--(A2111);
\draw (A21)--(A212);
\draw (A211)--(A2112);

\coordinate (B1) at (108:0.7cm);
\fill(B1) circle(2pt);
\draw(B1) circle(3pt);

\coordinate (B11) at (90:1.4cm);
\fill(B11) circle(2pt);
\draw(B11) circle(3pt);

\coordinate (B111) at (85:2.1cm);
\fill(B111) circle(2pt);
\draw(B111) circle(3pt);

\coordinate (B1111) at (78:3cm);
\fill(B1111) circle(2pt);
\draw(B1111) circle(3pt);

\coordinate (B1112) at (90:3cm);
\fill(B1112) circle(2pt);

\coordinate (B112) at (100:2cm);
\fill(B112) circle(2pt);

\coordinate (B1121) at (105:2.5cm);
\fill(B1121) circle(2pt);
\draw(B1121) circle(3pt);

\coordinate (B11211) at (100:3cm);
\fill(B11211) circle(2pt);
\draw(B11211) circle(3pt);

\coordinate (B11212) at (110:3cm);
\fill(B11212) circle(2pt);

\coordinate (B12) at (120:1.4cm);
\fill(B12) circle(2pt);

\coordinate (B121) at (125:2cm);
\fill(B121) circle(2pt);
\draw(B121) circle(3pt);

\coordinate (B1211) at (130:2.5cm);
\fill(B1211) circle(2pt);
\draw(B1211) circle(3pt);

\coordinate (B12111) at (125:3cm);
\fill(B12111) circle(2pt);
\draw(B12111) circle(3pt);

\coordinate (B12112) at (140:3cm);
\fill(B12112) circle(2pt);

\draw (O) -- (B1) --(B11)--(B111)--(B1111);
\draw (B111)--(B1112);
\draw(B11)--(B112)--(B1121)--(B11211);
\draw(B1121)--(B11212);
\draw (B1) --(B12)--(B121)--(B1211)--(B12111);
\draw (B1211)--(B12112);

\coordinate (C1) at (180:0.7cm);
\fill(C1) circle(2pt);
\draw(C1) circle(3pt);

\coordinate (C11) at (165:1.5cm);
\fill(C11) circle(2pt);
\draw(C11) circle(3pt);

\coordinate (C111) at (160:2.3cm);
\fill(C111) circle(2pt);
\draw(C111) circle(3pt);

\coordinate (C1111) at (155:3cm);
\fill(C1111) circle(2pt);
\draw(C1111) circle(3pt);

\coordinate (C1112) at (165:3cm);
\fill(C1112) circle(2pt);
\draw(C1112) circle(3pt);

\coordinate (C12) at (190:1.5cm);
\fill(C12) circle(2pt);

\coordinate (C121) at (180:3cm);
\fill(C121) circle(2pt);
\draw(C121) circle(3pt);

\coordinate (C122) at (195:2cm);
\fill(C122) circle(2pt);

\coordinate (C1221) at (190:3cm);
\fill(C1221) circle(2pt);
\draw(C1221) circle(3pt);

\coordinate (C1222) at (205:3cm);
\fill(C1222) circle(2pt);

\draw (O)--(C1)--(C11)--(C111)--(C1111);
\draw (C111)--(C1112);
\draw (C1)--(C12)--(C121);
\draw (C12)--(C122)--(C1221);
\draw (C122)--(C1222);

\coordinate (D1) at (240:1cm);
\fill(D1) circle(2pt);
\draw(D1) circle(3pt);

\coordinate (D11) at (235:1.7cm);
\fill(D11) circle(2pt);
\draw(D11) circle(3pt);

\coordinate (D111) at (230:2.3cm);
\fill(D111) circle(2pt);
\draw(D111) circle(3pt);

\coordinate (D1111) at (225:3cm);
\fill(D1111) circle(2pt);
\draw(D1111) circle(3pt);

\coordinate (D1112) at (233:3cm);
\fill(D1112) circle(2pt);

\coordinate (D112) at (245:2.3cm);
\fill(D112) circle(2pt);

\coordinate (D1121) at (240:3cm);
\fill(D1121) circle(2pt);
\draw(D1121) circle(3pt);

\coordinate (D1122) at (250:3cm);
\fill(D1122) circle(2pt);

\coordinate (D12) at (255:1.7cm);
\fill(D12) circle(2pt);

\coordinate (D121) at (260:2.3cm);
\fill(D121) circle(2pt);
\draw(D121) circle(3pt);

\coordinate (D1211) at (255:3cm);
\fill(D1211) circle(2pt);
\draw(D1211) circle(3pt);

\coordinate (D1212) at (265:3cm);
\fill(D1212) circle(2pt);

\coordinate (D2) at (270:1cm);
\fill(D2) circle(2pt);

\coordinate (D21) at (275:1.8cm);
\fill(D21) circle(2pt);
\draw(D21) circle(3pt);

\coordinate (D211) at (275:2.5cm);
\fill(D211) circle(2pt);
\draw(D211) circle(3pt);

\coordinate (D2111) at (272:3cm);
\fill(D2111) circle(2pt);
\draw(D2111) circle(3pt);

\coordinate (D2112) at (280:3cm);
\fill(D2112) circle(2pt);

\coordinate (D212) at (285:3cm);
\fill(D212) circle(2pt);

\draw (O)--(D1)--(D11)--(D111)--(D1111);
\draw (D111)--(D1112);
\draw (D11)--(D112)--(D1121);
\draw (D112)--(D1122);

\draw (D1)--(D12)--(D121)--(D1211);
\draw (D121)--(D1212);

\draw (O)--(D2)--(D21)--(D211)--(D2111);
\draw (D21)--(D212);
\draw (D211)--(D2112);

\coordinate (E1) at (324:1cm);
\fill(E1) circle(2pt);
\draw(E1) circle(3pt);

\coordinate (E11) at (310:1.8cm);
\fill(E11) circle(2pt);
\draw(E11) circle(3pt);

\coordinate (E111) at (300:2.5cm);
\fill(E111) circle(2pt);
\draw(E111) circle(3pt);

\coordinate (E112) at (320:2.3cm);
\fill(E112) circle(2pt);

\coordinate (E1121) at (315:3cm);
\fill(E1121) circle(2pt);
\draw(E1121) circle(3pt);

\coordinate (E1122) at (325:3cm);
\fill(E1122) circle(2pt);

\coordinate (E1111) at (295:3cm);
\fill(E1111) circle(2pt);
\draw(E1111) circle(3pt);

\coordinate (E1112) at (305:3cm);
\fill(E1112) circle(2pt);

\coordinate (E12) at (345:2cm);
\fill(E12) circle(2pt);

\coordinate (E121) at (340:3cm);
\fill(E121) circle(2pt);
\draw(E121) circle(3pt);

\coordinate (E122) at (355:3cm);
\fill(E122) circle(2pt);

\draw (O)--(E1)--(E11)--(E111)--(E1111);
\draw (E111)--(E1112);
\draw (E11)--(E112)--(E1121);
\draw (E112)--(E1122);
\draw (E1)--(E12)--(E121);
\draw (E12)--(E122);

\end{tikzpicture}
   
  \caption{A circular forest; firstborns are marked with an extra circle}
\label{fig:circularforest}
\end{figure}
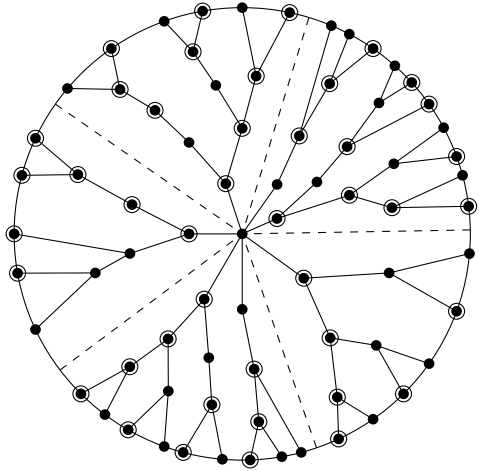

The \emph{weight} of a forest $F=(\tau_1,\tau_2,\dots,\tau_k)$ is the product
$$
\omega(F)=\prod_{\tau\in F} \omega(\tau)
.
$$
\end{defi}

 \begin{prop}
Assuming that $m$ and $k$ $(k\neq 0)$ have the same parity,
then
$$ 
p_{m,k}(n) = \sum_{F\in T^E_{m,n,k}}\omega(F)
.
$$
\end{prop}
\begin{proof}
From the proof of Theorem \eqref{th:traceformula}, we see that  
  \begin{align*}
   p_{m,k}(n)=&\sum_{\substack{
           l_0,l_1,\dots,l_k\geq 0\\
    l_1,\dots,l_{k-1},l_0+l_k\text{ even}  \\
    \sum l_i=m -k
         }}\Tr(B_n^{l_0}\J B_n^{l_1}\J B_n^{l_2}\dotsm \J B_n^{l_k})  
         \\&
         \sum_{\substack{
           l_0,l_1,\dots,l_k\geq 0\\
           l_1,\dots,l_{k-1},l_0+l_k\text{ even} 
           \\\sum l_i =m-k 
    }}
    \Tr(\J B_n^{l_0+l_k})
    \prod_{\substack{1\leq i\leq k-1}}\Tr(\J B_n^{l_i})
                                                                                                               \end{align*}
This can be visualized in terms of forests, see Figure~\ref{fig:circularforest}.
\end{proof}

\subsection{Interpretation of the constant term $\Tr(B_n^{2m})$}
The moment generating function of $B_n$ is
\begin{align*}
  M_{B_n}(z) &= \frac{n(1+z\tan(n\arctan z))}{1+z^2}=\frac{n+nz^2Q_n(z)}{1+z^2}.
 \intertext{If we expand the generating function in powers of $z$, then we obtain}
   \Tr(B_n^{2m})&=n d_{n,m-1}-\Tr(B_n^{2m-2})=n d_{n,m-1}-n d_{n,m-2}+n d_{n,m-3}+\dots+n d_{n,0}(-1)^{m-1}+n(-1)^m
   \\&=n\sum_{i=0}^{m}  d_{n,m-1-i}(-1)^{i}=n\sum_{i=0}^{m}  e_{n,m-i}(-1)^{i}
     \end{align*}

\section{Explicit analytic evaluation of cotangent sums}
In this section we study the Taylor series expansions of the generating
function \eqref{eq:genfunc} and obtain closed formulas in terms of derivative 
polynomials.

\subsection{A Möbius inversion}
As a first step we extract the coefficients of the generating
function using Fa\`a di Bruno's formula.
\begin{theo}\label{closedformformula}
 The cotangent sum \eqref{eq:cotsum}
 can be expressed as
  \begin{equation}
    \label{eq:Smna=sumtanpoly}
\cotsum(m,n,\alpha) = (-1)^{m/2}n \mathbbm{1}_{\text{$m$ even}} + \frac{
  (-i)^m}{(m-1)!}  \sum_{\nu\in\SPodd(m)}
   P_{\abs{\nu}-1}(\cot\alpha) (in)^{\abs{\nu}} \mu(\hat{0}_m,\nu)
\end{equation}
where 
$(-1)^nP_n(x)$ are the derivative polynomials for $\cot$ \eqref{eq:cotpoly},
$\SPodd(n)$ is the set of partitions with odd blocks only and
$\mu(\hat{0}_m,\nu)$  is the Möbius function of the partition lattice
\eqref{eq:moebius}.
\end{theo}

\begin{proof}
  We start by expressing the generating function \eqref{eq:genfunc} in terms of
the functions $f(z)=\ln(\abs{\sin(z-\alpha)})$ and $g(z)=n\arctan z$.
Indeed  observe that 
\begin{align*}
&\frac{n}{1+z^2}
    \left(
    1- z\cot(n\arctan z-\alpha)
    \right)
   =\frac{n}{1+z^2}
    -z\frac{d}{dz}f(g(z)).
    \end{align*}
and moreover the Leibniz rule of order $m$ implies
\begin{align*} 
&\frac{d^m}{dz^m}\left(z\frac{d}{dz}f(g(z))\right)=m\frac{d^m}{dz^m}f(g(z))+z\frac{d^{m+1}}{dz^{m+1}}f(g(z))
\intertext{thus }
&\frac{d^m}{dz^m}\left(z\frac{d}{dz}f(g(z))\right)\Bigg|_{z = 0} =m\frac{d^m}{dz^m}f(g(z))\Bigg|_{z = 0}.
\end{align*}
Now we can apply Fa\`a di Bruno's formula
\eqref{eq:faadibruno} for 
the $m$-th derivative of a composed function and obtain
 \begin{align*}
    \frac{d^m}{dz^m}f(g(z))
    &= \sum_{\nu\in \SP(m)}
        f^{(\abs{\nu})}(g(z))
        \prod_{B\in\nu}
         g^{(\abs{B})}(z)
      \\
    &= \sum_{\nu\in \SP(m)}
         \cot^{(\abs{\nu}-1)}(g(z)-\alpha)
         \prod_{B\in\nu}\frac{}{}
          \frac{ni(-1)^{\abs{B}}(\abs{B}-1)!}{2}
          ((z-i)^{-\abs{B}} - (z+i)^{-\abs{B}})
      \\
    &= \sum_{\nu\in \SP(m)}
         (-1)^{\abs{\nu}-1}P_{\abs{\nu}-1}(\cot(g(z)-\alpha))
         \left(
           -\frac{ni}{2}
         \right)^{\abs{\nu}}
         \mu(\hat{0}_m,\nu)
         \prod_{B\in\nu}\frac{}{}
          ((z-i)^{-\abs{B}} - (z+i)^{-\abs{B}})
  \end{align*}

  where we recognize the Möbius function of the partition lattice
  \eqref{eq:moebius}.
  Now at $z=0$ we have $g(0)=0$ and
  $$
  (-i)^{-k}-i^{-k} = i^k-(-i)^k = 
  \begin{cases}
    0 & \text{$k$ even}\\
    2i^k & \text{$k$ odd};
  \end{cases}
  $$
  moreover if $\nu$ is odd then $\abs{\nu}\equiv m \mod{ 2}$ 
  and we obtain
  \begin{equation*}
      \left.\frac{d^m}{dz^m}f(g(z))\right\rvert_{z=0}
    = -(-i)^m\sum_{\nu\in \SPodd(m)}
         P_{\abs{\nu}-1}(\cot \alpha )
         \left(
           ni
         \right)^{\abs{\nu}}
         \mu(\hat{0}_m,\nu)
       \end{equation*}
   finally the Taylor coefficients of $\frac{n}{1+z^2}$ contribute $ni^m$ for
   even $m$ and the claim follows.
       
\end{proof}

\begin{Rem}
From the Newton identities between power sum and elementary
symmetric polynomials we conclude
\begin{equation}
\sum_{l_1<l_2<\dots<l_k}\prod_{j=1}^k\cot\frac{\alpha+l_j\pi}{n}=(-1)^{-k}c_{n-k}=\begin{cases}
 \binom{n}{k} (-1)^{-k/2} &\text{$k$ even}\\
 \cot \alpha \binom{n}{k} (-1)^{(1-k)/2} &\text{$k$ odd}. \label{eq:identity}
\end{cases}
\end{equation}
In the case  $|B|=1$ this reduces to Theorem \ref{closedformformula} with $m=1$ 
and for  $|B|=n$ this is a consequence of the  well-known identities   $$\prod_{k=0}^{n-1}\sin\left(\frac{k\pi}{n}+z\right)=2^{1-n}\sin(nz)\text{ and }\prod_{k=0}^{n-1}\cos\left(\frac{k\pi}{n}+z\right)=2^{1-n}\sin\left(nz+\frac{\pi}{2}n\right).$$
 For further literature about trigonometric multiple cotangent sums similar to those  in \eqref{eq:identity}, we refer the reader to \cite[Section~6]{Berndt2002} and \cite{Zagier:1973}.
\end{Rem}

\subsection{An evaluation in terms of derivative polynomials}
We can apply Rota's calculus to further simplify
the expression \eqref{eq:Smna=sumtanpoly}.
\begin{cor}
  The cotangent sums \eqref{eq:cotsum} evaluate to
  \begin{equation}
    \label{eq:cotsum:rota}
    \cotsum(m,n,\alpha) = (-1)^{m/2}n \mathbbm{1}_{\text{$m$ even}}
    + \frac{1}{(m-1)!}   \sum_{k=1}^m
   n^k A_{m}^{(k)} \, P_{k-1}(\cot\alpha)
 \end{equation}
 where $A_m^{(k)}$ are the arctangent numbers
  \eqref{eq:def-arctannumbers};
  note that these are alternating   \eqref{eq:signarctannumbers}.
\end{cor}
\begin{proof}
  We extract the essential part of the formula \eqref{eq:Smna=sumtanpoly}
  and arrive at the expression
  $$
  \sum_{\nu\in\SPodd(m)} P_{\abs{\nu}-1}(\cot\alpha)(ni)^{\abs{\nu}}\mu(\hat{0}_m,\nu)
  = \sum_{k=1}^m c_{m,k} (ni)^kP_{k-1}(\cot\alpha)
  $$
  where
  $$
  c_{m,k} =
  \sum_{\substack{\nu\in\SPodd(m)\\
      \abs\nu = k
    }} \mu(\hat{0}_m,\nu)
  $$
  This sum can be evaluated using the combinatorial convolution
  \eqref{eq:faadibrunoconvolution} by setting
  \begin{equation*}
    f_k=
         \begin{cases}
           (k-1)!& \text{for odd $k$} \\
           0 & \text{else}
         \end{cases}
  \end{equation*}
  and $g_k=t^k$
  and the generating functions are
  $$
  F_f(z) =\sum_{k \text{ odd}} \frac{(k-1)!}{k!} z^k =
  \frac{1}{2}(\log(1+z)-\log(1-z)) = \frac{1}{2}\log\frac{1+z}{1-z}
  =\atanh z 
  $$
  and
  $$
  F_g(z) = \sum_{k=1}^\infty \frac{t^k}{k!}z^k = e^{tz}-1
  ;
  $$ 
  hence by   \eqref{eq:faadibrunocomposition}
  $$
  F_g(F_f(z)) = e^{t\atanh z}-1
  $$
  and the coefficient of $t^k$ yields the desired coefficient
  $c_{m,k}= \tilde{A}_m^{(k)}$ and from   \eqref{eq:signarctannumbers}
  we gather the correct sign.
\end{proof}

\begin{Rem}
  \label{rem:cvijovicformula}
  Comtet \cite[p.~260]{Comtet} asserts that
  the arctangent numbers are inverse to the derivative polynomials.
  This means that the standard monomials can be expanded as 
  a linear combination of tangent polynomials as follows
  \cite[Formula (2.14)]{Cvijovic:2011:higher}:
\begin{equation}
  \label{eq:xm=sumPk}
  x^m = \frac{1}{(m-1)!} \sum_{k=1}^m A_m^{(k)} P_{k-1}(x) + (-1)^{m/2} \mathbbm{1}_{\text{$m$ even}}
\end{equation}
Let us explain now that the similarity of this formula with
\eqref{eq:cotsum:rota} is not a coincidence.
Indeed using the property that the derivative polynomials
linearize the cotangent power
and the simple formula $\cotsum(1,n,\alpha) = \Tr C_n=n\cot\alpha$ 
allow for the following alternative straightforward proof:

\begin{align*}
  \sum_{s=1}^n \cot^m\frac{\alpha+s\pi}{n} 
  &=  n(-1)^{m/2} \mathbbm{1}_{\text{$m$ even}}
      +
      \frac{1}{(m-1)!}
      \sum_{s=1}^n 
        \sum_{k=1}^m A_m^{(k)} P_{k-1}(
       \cot\frac{\alpha+s\pi}{n} 
       )  
\\
  &=  n(-1)^{m/2} \mathbbm{1}_{\text{$m$ even}}
   +
   \frac{1}{(m-1)!}
   \sum_{s=1}^n 
   \sum_{k=1}^m A_m^{(k)} n^{k-1} (-1)^{k-1}\frac{d^{k-1}}{d\alpha^{k-1}}  \cot\frac{\alpha+s\pi}{n}  
  \\
  &=
  n(-1)^{m/2} \mathbbm{1}_{\text{$m$ even}}
  +
  \frac{1}{(m-1)!} \sum_{k=1}^m A_m^{(k)} n^{k-1}
   (-1)^{k-1}\frac{d^{k-1}}{d\alpha^{k-1}} 
  \sum_{s=1}^n 
 \cot\frac{\alpha+s\pi}{n}   
 \\
  &=
  n(-1)^{m/2} \mathbbm{1}_{\text{$m$ even}}
  +
  \frac{1}{(m-1)!} \sum_{k=1}^m A_m^{(k)} n^{k-1}
   (-1)^{k-1}\frac{d^{k-1}}{d\alpha^{k-1}} 
    n\cot\alpha  
  \\
  &=
  n(-1)^{m/2} \mathbbm{1}_{\text{$m$ even}}
  +
  \frac{1}{(m-1)!} \sum_{k=1}^m A_m^{(k)} n^k
  P_{k-1}(\cot\alpha)  
\end{align*}

\end{Rem}

\begin{Rem}
We are grateful to an anonymous referee who brought the
paper \cite {LvShen:2017:chebyshev} to our attention,
where an inverse relation for Chebyshev Polynomials
analogous to \eqref{eq:xm=sumPk}  is used
to evaluate cosine power sums with proof similar to 
Remark~\ref{rem:cvijovicformula}.
\end{Rem}

\subsection{Special cases}

Let us evaluate formula 
    \eqref{eq:cotsum:rota} at certain offsets.
We start with the elementary evaluations at $\alpha=\pi/2$ and $\alpha=\pi/4$.
The first sum yields the constant coefficient $p_{m,0}(n)$ from
Theorem~\ref{th:traceformula}
it vanishes for odd $m$ and equals the free cumulants of the
generalized tetilla law, 
see \cite[Proposition~4.10]{EjsmontLehner:2019:commutators}.
The remaining coefficients are computed in Corollary~\ref{cor:pmrn} below.
The second sum provides an explicit formula for the sums considered by
Byrne and Smith \cite{ByrneSmith:1997:integer}.
\begin{cor}
  \label{cor:cotsumpi/24}
  \begin{align}
    \label{eq:cotsum:pi/2}
    \cotsum(2m,n,\pi/2) 
    & = (-1)^{m}n
    + \frac{1}{(2m-1)!}   \sum_{k=1}^m n^{2k} A_{2m}^{(2k)}\,T_{2k-1}
    \\
    \cotsum(m,n,\pi/4) 
    & = (-1)^{m/2}n \mathbbm{1}_{\text{$m$ even}}
    + \frac{1}{2(m-1)!}   \sum_{k=1}^m (2n)^k A_{m}^{(k)} \, E_{k-1}
    \nonumber
 \end{align}
\end{cor}

\begin{proof}
  The evaluation of the generating function \eqref{eq:derivative:genfunc} yields
  $$
  P(0,z) = \tan z
  \qquad
  P(1,z) = \frac{1+\tan z}{1-\tan z} = \tan(2z)+\sec(2z)
  $$
  and we conclude that
  $P_n(0)=T_n$ and the expansion
  \eqref{eq:Pnx=sumTnkxk} at $x=1$ 
  yields the identity  \cite[(10)]{KnuthBuckholtz:1967} 
  $$
  2^nE_n= P_n(1) = T_n + \sum_{k=1}^{n+1} \frac{T_{n+1}^{(k)}}{k} 
  =:E_n^B
  $$
  where $E_n^B$ are also known as
  Euler numbers of type B \cite{Ma:2014:gamma}.
  See \cite[Remark~4.6]{EjsmontLehner:2020:tangent}
  for another, similar identity for $E_n^B$.
\end{proof}

\subsection{Explicit evaluation of the coefficients $p_{m,r}$}
We observed in  Theorem~\ref{th:traceformula}
that the cotangent sum \eqref{eq:cotsum} can be expressed
as a bivariate polynomial in $n$ and $\cot\alpha$ with rational
coefficients.
We can identify these coefficients explicitly
by applying the expansion 
\eqref{eq:Pnx=sumTnkxk}  to the evaluation \eqref{eq:cotsum:rota}.
\begin{cor}
  \label{cor:pmrn}
  The integer valued polynomials $p_{m,r}(n)$
  appearing in \eqref{eq:polynomialcoeff}
  have the following explicit expressions
  \begin{align*}
    p_{2m,0}(n) &= \cotsum(2m,n,\pi/2)\\
    p_{m,r}(n) &= \frac{1}{r(m-1)!}\sum_{k=r}^m n^kA_m^{(k)}T_k^{(r)}
    .
  \end{align*}
\end{cor}
\subsection{Evaluation of the sum of Berndt and Yeap}
Finally let us give an alternative and somewhat simpler expression for
the summation formula of Berndt and Yeap 
 \cite[Corollary 2.2]{Berndt2002} 
 \begin{equation*} 
   \sum_{k=1}^{n-1} \cot^{2m}\frac{k\pi}{n} 
   = (-1)^mn - (-1)^m2^{2m}\sum_{\substack{j_0,j_1,j_2,\dots ,j_{2m}\geq 0\\
j_0+j_1+j_2+\dots +j_{2m}=m    
    }}n^{2j_0}\prod_{p=0}^{2m}\frac{B_{2j_p}}{(2j_p)!}. 
    .
\end{equation*}

\begin{cor}
  \label{cor:BerndtYeap}
    The sum $S_0(2m,n)$  can be evaluated as follows
    \begin{equation*}
      \sum_{k=1}^{n-1} \cot^{2m}\frac{k\pi}{n}
      = (-1)^m (n-1) - \frac{1}{(2m-1)!}\sum_{k=1}^{m}  (-1)^kA_{2m}^{(2k)} \frac{4^kB_{2k}}{2k}(n^{2k}-1).
    \end{equation*}
\end{cor}

\begin{proof}
  The sum $S_0(2m,n)$ can be obtained from the general formula
  $\cotsum(2m,n,\alpha)$
  after removing the singular term at $k=0$ and
  then taking the limit $\alpha\to 0$.
  Let 
  $$
  S_0(2m,n,\alpha) = 
\sum_{k=1}^{n-1} \cot^{2m}\frac{\alpha+k\pi}{n}
=\cotsum(2m,n,\alpha)-\cot^{2m}\frac{\alpha}{n}
,
  $$
  then $S_0(2m,n) = \lim_{\alpha\to0}S_0(2m,n,\alpha)$.
  First we linearize the singular term according to formula
  \eqref{eq:xm=sumPk} and combine it with the summation formula
  \eqref{eq:cotsum:rota} to obtain
  $$
  S_0(2m,n,\alpha)
  = (-1)^m(n-1)
    + \frac{1}{(2m-1)!} 
      \sum_{k=1}^{2m}
      A_{m}^{(k)}(   n^k  \, P_{k-1}(\cot\alpha) - P_{k-1}(\cot(\alpha/n))
      .
  $$
  Next we replace the polynomial evaluation by the derivative
  according to \eqref{eq:cotpoly} and we see that
  $$
  n^{k+1} P_k(\cot \alpha) -    P_k(\cot (\alpha/n))
  = (-1)^k(n^{k+1}\cot^{(k)}(\alpha) - \cot^{(k)}(\alpha/n))
  $$

  At this point it is convenient to recall the series expansion of cotangent
  $$
  \cot z = \frac{1}{z} + \sum_{p=1}^\infty (-1)^p \frac{2^{2p}B_{2p}}{(2p)!}z^{2p-1}
  $$
  to observe that the derivatives of the singular term $1/z$ cancel
  and we can express the difference in terms of the analytic part
  $$
  \gamma(z) = \cot z -\frac{1}{z} =  \sum_{p=1}^\infty (-1)^p \frac{2^{2p}B_{2p}}{(2p)!}z^{2p-1}
  $$
  and find
  \begin{align*}
  \lim_{\alpha\to0}   n^{k+1} P_k(\cot \alpha) -    P_k(\cot (\alpha/n))
  &= (-1)^k\lim_{\alpha\to0}   n^{k+1} \gamma^{(k)}(\alpha) -    \gamma^{(k)} (\alpha/n)) \\
  &= (-1)^k(n^{k+1}-1)\gamma^{(k)}(0)\\
  &=
  \begin{cases}
    0 & \text{$k$ even}\\
    -(-1)^{(k+1)/2}(n^{k+1}-1)\frac{2^{k+1}B_{k+1}}{k+1} & \text{$k$ odd}\\
  \end{cases}
  \end{align*}
  and finally
  $$
  \lim_{\alpha\to0}  S_0(2m,n,\alpha)
  = (-1)^m(n-1)
    - \frac{1}{(2m-1)!} 
      \sum_{k\text{ even}}
      (-1)^{k/2}A_{m}^{(k)}(   n^k  -1) \frac{2^kB_k}{k}
      .
  $$
\end{proof}

\begin{Rem}
The generating function of 
the Euler zigzag numbers   \eqref{eq:def:zigzag}
is  related to the generating
function \eqref{eq:derivative:genfunc}
$$
\tan(z)+\sec(z)= \frac{1+\tan(z/2)}{1-\tan(z/2)} = \sum_{n=0}^\infty P_n(1)\frac{z^n}{2^nn!}
$$
and comparing with the explicit formula for the derivative polynomials 
\eqref{eq:tanpoly} we conclude the following identity:
$$ {E_{n}} =-(-i)^{n}\sum_{k=0}^{n}\frac{k!}{2^k}\stirling{n}{k}(i-1)^{k+1}.$$
See \cite{Cvijovic:2009:values} for other evaluations of the derivative
polynomials at rational angles.
\end{Rem}

\begin{cor}Extracting the linear coefficient
of \eqref{eq:cotsum:rota} we can obtain an explicit expression
for the moments     \eqref{eq:def-dnm}
  \label{cor:TrJBmexpl}
  \begin{equation*}     \Tr(\J B_n^{2m})
    = \frac{1}{(2m)!}
    \left(
      A_{m+1}^{(1)} n + \sum_{k=1}^{2m+1} T_{k-1}^{(2)} A_{2m+1}^{(k)} n^k
    \right)
  \end{equation*}
  where $T_n^{(k)}$ are the higher tangent numbers \eqref{eq:tangentnumbers}.
\end{cor}

\section{Concluding Remarks}
 In this section we connect the algebraic and analytic approach and  give some final remarks. 
 \subsection{Another explicit formula for $\alpha=\frac{\pi}{2}$}\label{sec:alpphapitow}
  From \cite[ Problem 76 on P. 317, Answer on P. 559]{GrahamKnuthPatashnik:1994}, we infer the identity (cf.~\cite[(3.29)] {Boyadzhiev:2005})
  $$
  \omega_n
  \bigl(
  -1/2
  \bigr)
  =
  \sum_{k=1}^{m}(-1)^k\frac{k!}{2^k}\left\{{{m} \atop{k}}\right\}=
\left\{ \begin{array}{ll}
\frac{2}{m+1}(1-2^{m+1})B_{m+1}  & \textrm{if $m$ is odd} \\
0 & \textrm{if $m$ is even.}
\end{array} \right.$$
If we plug in  $\alpha=\frac{\pi}{2}$ into Equation \eqref{eq:Smna=sumtanpoly} we will take into account the equation \eqref{eq:tanpoly}
then for $m$ even (for $m$ odd the sum is zero) our sums can be written  in terms of Bernoulli numbers (which 
frequently appear in trigonometric sums, see \cite{Berndt2002,Cvijovic2009,Annaby2011,Fonseca2017,He:2020})
\begin{align*}\cotsum(m,n,{\pi}/{2})= (-1)^{m/2}n + \sum_{\substack{\nu\in \SPodd(m)\\|\nu| \text{ is even}}}&\frac{(-1)^{m/2}\nu!(2n)^{|\nu|}} {(m-1)!|\nu|}(1-2^{|\nu|})B_{|\nu|}.
\end{align*}

\subsection{Asymptotic analysis and derivative}
 In order to investigate asymptotic properties formula from Theorem \ref{closedformformula} it is sufficient to
consider the contribution of the singleton partition and we obtain  
\begin{align*}
\lim_{n\to \infty}\frac{1}{n^m}\sum_{k=0}^{n-1} \cot^m\frac{\alpha+k\pi}{n}
&=\left\{ \begin{array}{ll}
 \frac{
  1}{(m-1)!}  
   P_{m-1}(\cot\alpha)    & \textrm{if $m>1$} \\
\cot \alpha & \textrm{if $m =1$.}
\end{array} \right.
  \end{align*}
In particular from equation \eqref{eq:traceformula} we infer
the asymptotic expression
\begin{align*}  
\frac{
  1}{(m-1)!}  
   P_{m-1}(z)=\lim_{n\to\infty} \Tr \left[ \left(z \left[\begin{smallmatrix}
    1/n & 1/n\\
    1/n & 1/n
    \end{smallmatrix}\right]_n+\left[\begin{smallmatrix}
    0 & i/n\\
    -i/n& 0
    \end{smallmatrix}\right]_n
\right)^m \right]\text{ for }m>1. 
    \end{align*} 
Similarly  we prove that the derivatives of tangent and cotangent can be approximated by simple matrices. 
                
Finally we examine the  limit formula for $\alpha=\frac{\pi}{2}$.
From Section \ref{sec:alpphapitow} we conclude
\begin{align*}\lim_{n\to \infty}\frac{1}{n^m}\sum_{k=0}^{n-1} \cot^m\frac{\frac{\pi}{2}+k\pi}{n}&=
\left\{ \begin{array}{ll}
\frac{(-1)^{m/2+1}2^{m}(2^{m}-1)B_{m}} {m!} & \textrm{if $m$ is even} \\
0 & \textrm{if $m$ is odd.}
\end{array} \right.
\end{align*}

Indeed inspecting formula
\eqref{eq:cotsum:pi/2}
immediately yields
the asymptotics
\begin{equation*}
  \sum_{k=0}^{n-1} \cot^{2m}\left(\frac{\pi}{2n}+\frac{k\pi}{n}\right)
  =(-1)^{m+1}A_{2m}^{(2m)}(2^{2m}-1)n^{2m}\frac{2^{2m}B_{2m}}{(2m)!} +
  \mathcal{O}(n^{2m-2})
\end{equation*}
and since $A_{2m}^{(2m)}=1$ this  yields the desired limit.

Euler's identity $\zeta(2k)=\frac{(-1)^{k+1}(2\pi)^{2k}{B_{2k}}}{2(2k)!}$ 
and the preceding discussions give rise
to a new approximation of the values of the Riemann zeta function
at  even integer arguments, namely 
\begin{align*}
              \zeta(2k)=\lim_{n\to\infty}\frac{\pi^{2k}\Tr\big( \left[\begin{smallmatrix}
    0 &i\\
    -i& 0
    \end{smallmatrix}\right]_n^{2k}\big)}{2n^{2k}(2^{2k}-1)}\quad\text{for }k\in\N.
    \end{align*} 
Approximation of the Riemann zeta function for even values by powers of cotangent is well studied, see \cite{Yaglom:1967:challenging2,Williams1971,Apostol1973,Cvijovic2003}.

\bibliographystyle{amsplain}

\providecommand{\bysame}{\leavevmode\hbox to3em{\hrulefill}\thinspace}
\providecommand{\MR}{\relax\ifhmode\unskip\space\fi MR }
\providecommand{\MRhref}[2]{  \href{http://www.ams.org/mathscinet-getitem?mr=#1}{#2}
}
\providecommand{\href}[2]{#2}

\end{document}